\documentclass{amsart}
\usepackage{amscd,amsmath,amssymb,amsthm}
\usepackage{color}

\theoremstyle{plain}
\newtheorem{thm}{Theorem}[section]
\newtheorem{cor}[thm]{Corollary}
\newtheorem{lem}[thm]{Lemma}
\newtheorem{prop}[thm]{Proposition}
\newtheorem{definition}[thm]{Definition}

\newtheorem{remark}[thm]{Remark}

\newcommand\dist{{\operatorname{dist}}}

\newcommand\Index{{\operatorname{Index}}}

\newcommand\id{{\operatorname{id}}}
\newcommand{\eps}{\varepsilon}

\title[Two Families of Examples of Groups Acting on Trees]{Two Families of Examples of Groups Acting on Trees with Nontrivial Quasi-Kernels}
\author{Nikolay A. Ivanov}
\address{Faculty of Mathematics and Informatics\\University of Sofia\\blvd.\ James Bourchier 5\\BG-1164 Sofia\\Bulgaria}
\email{nivanov@fmi.uni-sofia.bg}

\begin{document}

\begin{abstract} 
We introduce two families of examples of groups acting on trees, one consisting of group amalgamations and the other consisting of HNN-extensions, 
motivated by the problems of $C^*$-simplicity and unique trace property. Moreover, we prove that our examples are not inner amenable 
and identify a relatively large, simple, normal subgroup in each one. 
\end{abstract} 

\maketitle

\section{Introduction} 

The questions of $C^*$-simplicity and unique trace property for a discrete group have been studied extensively. By definition, a discrete group $G$ 
is $C^*$-simple if the $C^*$-algebra associated to the left regular representation, $C_r^*(G)$, is simple, likewise it has the unique trace property if 
$C_r^*(G)$ has a unique tracial state. An extensive introduction to that topic was given by de la Harpe (\cite{Harpe}). Recently, Kalantar and 
Kennedy (\cite{KK}) gave a necessary and sufficient condition for $C^*$-simplicity in terms of action on the Furstenberg boundary. Later, 
Breuillard, Kalantar, Kennedy, and Ozawa (\cite{BKKO}) studied further the question of $C^*$-simplicity and also showed that a group has the unique 
trace property if and only if its amenable radical is trivial. They also showed that $C^*$-simplicity implies the unique trace property. The reverse 
implication was disproven by examples given by Le Boudec (\cite{leboudec17}). 
\par 
The notion of inner amenability for discrete groups was introduced by Effros (\cite{effros}) as an analogue to Property $\Gamma$ for $II_1$ factors, 
introduced by Murray and von Neumann (\cite{MvN}). By definition, a discrete group $G$ is inner amenable if there exist a conjugation invariant, 
positive, finitely additive, probability measure on $G \setminus \{ 1 \}$. Effros showed that Property $\Gamma$ implies inner amenability, but the 
reverse implication doesn't hold, as demonstrated in an example of Vaes (\cite{vaes}).  
\par 
In section 2 we give some preliminaries on group amalgamations, HNN-extensions, groups acting on trees, inner amenability, etc. We also prove an 
auxiliary statement about inner amenability, using equivariant maps (\cite{BH86}), which is interesting in its own right. 
In sections 3 and 4 we give our examples on group amalgamations and HNN-extensions, respectively. 
The structures of these two sections are alike. First, we give some basic properties; then, we study the group-theoretic 
structure; finally, we study the analytic structure: unique trace property, $C^*$-simplicity, and inner amenability. 
\par
Finally, we point out that the examples of section 3 generalize the example given in \cite[Section 4]{IO}, and the examples of section 4 generalize 
the example given in \cite[Section 5]{BIO}.

\section{Technical Details and Preliminaries}

For a group $\Gamma$ acting on a set $X$, we denote the set-wise stabilizer of a subset $Y \subset X$ by
$$\Gamma_{ \{ Y \} } \ \equiv \ \{\ g \in \Gamma \ | \ gY = Y \ \}$$
and the point-wise stabilizer of a subset $Y \subset X$ by
$$\Gamma_{ ( Y ) } \ \equiv \ \{\ g \in \Gamma \ | \ gy = y, \ \forall y \in Y \ \}.$$ 
For a point $x \in X$, we denote its stabilizer by 
$$\Gamma_{ x } = \{\ g \in \Gamma \ | \ gx = x \ \}.$$ 
Note that, $\Gamma_{ \{ Y \} }$, $\Gamma_{(Y)}$, and $\Gamma_x$ are all subgroups of $\Gamma$. Also note that, 
$$ g \Gamma_{ \{ Y \} } g^{-1} = \Gamma_{ \{ gY \} }, \ g \Gamma_x g^{-1} = \Gamma_{gx} 
\ \text{, and} \ g \Gamma_{ ( Y ) } g^{-1} = \Gamma_{ ( gY ) }.$$ 
\par 
For a group $G$ and its subgroup $H$, by $\langle \langle H \rangle \rangle_G$ or by $\langle \langle H \rangle \rangle$ we will denote the 
normal closure of $H$ in $G$. 
\par 
Some general references on group amalgamations and HNN-extensions are, e.g., \cite{baumslag}, \cite{cohen}, \cite{Serre}, \cite{Harpepreaux}, etc. 
\par 
Let $G_i \ = \ \langle X_i \ | \ R_i \rangle$ for $i = 0,1$ be two groups having a common subgroup $H$ embedded via 
$j_i : H \hookrightarrow G_i$. Their free product with amalgamation is the group 
$$G_0 *_H G_1 \ \equiv \ \langle X_0 \sqcup X_1 \ | \ R_0 \sqcup R_1 \sqcup \{ j_0(h) = j_1(h) \ | \ h \in H \} \rangle.$$ 
We will assume that the embeddings $j_i$ are self-evident. Every element $g \ \in \ G_0 *_H G_1 \setminus H$ can be written in reduced form as 
$$g \ = \ (g_0) g_1 \cdots g_n, \text{  where  } n \in \mathbb{N}_0 \text{  and  } g_k \in G_{k \pmod 2} \setminus H.$$ 
If $S_i$ is a set of left coset representatives for $G_i / H$, where $i = 0,1$, satisfying $S_0 \cap S_1 = \{ 1 \}$, then every element 
$g \ \in \ G_0 *_H G_1$ can be uniquely written in normal form as 
$$g \ = \ (s_0) s_1 \cdots s_n \cdot h, \text{  where  } n \in \mathbb{N}_0, \ \ s_k \in S_{k \pmod 2} \setminus \{ 1 \}, \text{  and  } h \in H.$$ 
The group amalgamation $G_0 *_H G_1$ is called nontrivial if $G_0 \not= H \not= G_1$ and is called nondegenerate if, moreover, 
$\Index[G_0:H] \geq 3$ or $\Index[G_1:H] \geq 3$. \\ 
The Bass-Serre tree $T[G_0 *_H G_1]$ of $G_0 *_H G_1$ is the graph, that can be shown to be a tree, consisting of a vertex set 
$$\text{Vertex}(T[G_0 *_H G_1]) \ = \ \{ G_0 \} \cup 
\{ (s_0) s_1 \cdots s_n G_{n+1 \pmod 2} \ | \ n \in \mathbb{N}_0, \ s_k \in S_{k \pmod 2} \setminus \{ 1 \} \}$$ 
and an edge set 
$$\text{Edge}(T[G_0 *_H G_1]) \ = \ \{ (s_0) s_1 \cdots s_n H \ | \ n \in \mathbb{N}_0, \ s_k \in S_{k \pmod 2} \setminus \{ 1 \} \}.$$ 
In this situation, the vertex $(s_0) s_1 \cdots s_n G_{n+1 \pmod 2}$ is adjacent to the vertex $(s_0) s_1 \cdots s_n s_{n+1} G_{n \pmod 2}$ with 
connecting edge 
$$(s_0) s_1 \cdots s_n G_{n+1 \pmod 2} \ \cap \ (s_0) s_1 \cdots s_n s_{n+1} G_{n \pmod 2} \ = \ (s_0) s_1 \cdots s_n s_{n+1} H.$$ 
$G_0 *_H G_1$ acts on $T[G_0 *_H G_1]$ by left multiplication. \\ 
Finally, $G_0 *_H G_1$ has the following universal property (see, e.g., \cite{cohen}, page 29 or \cite{baumslag}, page 128): 
\begin{remark} \label{amalgamuniversal} 
Let $C$ be a group and let $\alpha_i : G_i \longrightarrow C$ be group homomorphisms ($i = 0,1$) for which $\alpha_0 |_H = \alpha_1 |_H$. 
Then, there is a unique group homomorphism $\beta : G_0 *_H G_1 \longrightarrow C$ satisfying $\beta |_{G_i} = \alpha_i$ for every $i =0,1$.
\end{remark}
\par
Let $G \ = \ \{ X \ | \ R \}$ be a group, let $H$ be a subgroup of $G$, and let $\theta : H \hookrightarrow G$ be a monomorphism. Then, an 
HNN-extension of this data (named after G. Higman, B. Neumann, H. Neumann) is the group 
$$HNN(G, H, \theta) \ \equiv \ G *_\theta \ \equiv \ 
\langle X \sqcup \{ \tau \} \ | \ R \sqcup \{ \theta(h) \ = \ \tau^{-1} h \tau \ | \ h \in H \} \rangle.$$ 
It is convenient to denote $H_{-1} \ \equiv \ H$ and $H_1 \ \equiv \ \theta(H)$. 
Every element $\gamma \in \ HNN(G, H, \theta)$ can be written in reduced form as 
\begin{multline*}
\gamma \ = \ g_1\tau^{\eps_1}\cdots g_n\tau^{\eps_n}g_{n+1}, \text{  where  } n \in \mathbb{N},\ g_1, \dots, g_{n+1} \ \in \ G,\ 
\eps_1, \dots, \eps_n \ = \ \pm 1, \\ 
\text{and where if  } \eps_{i+1}=-\eps_i \text{ for } 1\leq i\leq n-1, \text{ then } g_{i+1}\notin H_{\eps_i}. 
\end{multline*}
If $S_\eps$ is a set of left coset representatives for $G / H_\eps$, where $\eps = \pm 1$, satisfying $S_{-1} \ \cap \ S_1 \ = \ \{ 1 \}$, 
then every element $\gamma \in \ HNN(G, H, \theta)$ can be uniquely written in normal form as 
\begin{multline*} 
\gamma \ = \ s_1\tau^{\eps_1}s_2\tau^{\eps_2}\cdots s_n\tau^{\eps_n}g, \text{  where  } n \in \mathbb{N}_0, \ g \in G, \ \eps_i = \pm 1, 
\ s_i \in S_{-\eps_i},\ \forall 1 \leq i \leq n, \\ 
\text{and where if } \eps_{i-1} = -\eps_i \text{ for } 2 \leq i \leq n, \text{ then } s_i  \not= 1. 
\end{multline*}
The HNN-extension $HNN(G, H, \theta)$ is called nondegenerate if either $H \not= G$ or $\theta(H) \not= G$ and is called non-ascending if 
$H \not= G \not= \theta(G)$. \\ 
The Bass-Serre tree $T(HNN(G, H, \theta))$ of $HNN(G, H, \theta)$ is the graph, that can be shown to be a tree, consisting of a vertex set 
\begin{multline*} 
\text{Vertex}(HNN(G, H, \theta)) \ = \\ 
\{ G \} \ \cup \ \{ s_1\tau^{\eps_1}s_2\tau^{\eps_2}\cdots s_n\tau^{\eps_n} G \ | \ 
n \in \mathbb{N}, \ \ s_1\tau^{\eps_1}s_2\tau^{\eps_2}\cdots s_n\tau^{\eps_n} \text{ is in normal form} \}
\end{multline*} 
and edge set 
\begin{multline*} 
\text{Edge}(HNN(G, H, \theta)) \ = \\ 
\{ H \} \ \cup \ \{ s_1\tau^{\eps_1}s_2\tau^{\eps_2}\cdots s_n\tau^{\eps_n} H \ | \ 
n \in \mathbb{N}, \ \ s_1\tau^{\eps_1}s_2\tau^{\eps_2}\cdots s_n\tau^{\eps_n} \text{ is in normal form} \}.
\end{multline*}
In this situation, the vertex $v = s_1 \tau^{\eps_1} s_2 \tau^{\eps_2} \cdots s_n \tau^{\eps_n} G$ is adjacent to the vertex 
$w = s_1 \tau^{\eps_1} s_2\tau^{\eps_2} \cdots s_n \tau^{\eps_n} s_{n+1} \tau^{\eps_{n+1}} G$ with connecting edge 
\begin{equation*} 
e \ = \  
\begin{cases}
s_1\tau^{\eps_1}s_2\tau^{\eps_2}\cdots s_n\tau^{\eps_n} s_{n+1} \tau^{\eps_{n+1}} H \text{   if   } \eps_{n+1} = -1, \\ 
s_1\tau^{\eps_1}s_2\tau^{\eps_2}\cdots s_n\tau^{\eps_n} s_{n+1} H \text{   if   } \eps_{n+1} = 1.
\end{cases}
\end{equation*}
To see the reason for this, we need to look at the stabilizers. The stabilizer of $v$ is 
$$HNN(G, H, \theta)_v \ = \ s_1\tau^{\eps_1}s_2\tau^{\eps_2}\cdots s_n\tau^{\eps_n} G 
(s_1\tau^{\eps_1}s_2\tau^{\eps_2}\cdots s_n\tau^{\eps_n})^{-1},$$ 
and the stabilizer of $w$ is 
$$HNN(G, H, \theta)_w \ = \ s_1\tau^{\eps_1}s_2\tau^{\eps_2}\cdots s_n\tau^{\eps_n} s_{n+1} \tau^{\eps_{n+1}} G 
(s_1\tau^{\eps_1}s_2\tau^{\eps_2}\cdots s_n\tau^{\eps_n} s_{n+1} \tau^{\eps_{n+1}})^{-1}.$$ 
Therefore, the stabilizer of $e$ is 
\begin{multline*}
HNN(G, H, \theta)_e \ = \ HNN(G, H, \theta)_v \ \cap \ HNN(G, H, \theta)_w \ = \\ 
s_1\tau^{\eps_1}s_2\tau^{\eps_2}\cdots s_n\tau^{\eps_n} s_{n+1} \ [G \ \cap \ \tau^{\eps_{n+1}} G \tau^{-\eps_{n+1}}] \ 
(s_1\tau^{\eps_1}s_2\tau^{\eps_2}\cdots s_n\tau^{\eps_n} s_{n+1})^{-1} \ = \\ 
s_1\tau^{\eps_1}s_2\tau^{\eps_2}\cdots s_n\tau^{\eps_n} s_{n+1} \ H_{-\eps_{n+1}} \ 
(s_1\tau^{\eps_1}s_2\tau^{\eps_2}\cdots s_n\tau^{\eps_n} s_{n+1})^{-1} \ = \\ 
\begin{cases}
s_1\tau^{\eps_1}s_2\tau^{\eps_2}\cdots s_n\tau^{\eps_n} s_{n+1} H 
(s_1\tau^{\eps_1}s_2\tau^{\eps_2}\cdots s_n\tau^{\eps_n} s_{n+1})^{-1} \text{   if   } \eps_{n+1} = 1, \\ 
s_1\tau^{\eps_1}s_2\tau^{\eps_2}\cdots s_n\tau^{\eps_n} s_{n+1} \tau^{\eps_{n+1}} H \tau^{-\eps_{n+1}} 
(s_1\tau^{\eps_1}s_2\tau^{\eps_2}\cdots s_n\tau^{\eps_n} s_{n+1})^{-1} \text{   if   } \eps_{n+1} = -1.
\end{cases}
\end{multline*} 
$HNN(G, H, \theta)$ acts on $T(HNN(G, H, \theta))$ by left multiplication. \\ 
Finally, since $HNN(G, H, \theta)$ can be expressed as 
$$HNN(G, H, \theta) \ = \ ( G * \langle \tau \rangle ) / \langle \langle \tau^{-1} h \tau \theta(h^{-1}) \ | \ h \in H \rangle \rangle,$$ 
it has the following universal property (see, e.g., \cite{cohen}, page 36): 
\begin{remark} \label{HNNuniversal}
Let $C$ be a group, let $\alpha : G \longrightarrow C$ be a group homomorphism, and let $t \in C$ be an element for which the following holds: 
$t^{-1} \alpha(h) t = \alpha(\theta(h))$ for each $h \in H$. Then, there is a unique group homomorphism 
$\beta : HNN(G, H, \theta) \longrightarrow C$ satisfying $\beta |_G = \alpha$ and $\beta(\tau) = t$.
\end{remark}
\par
The barycentric subdivision of a graph $T$ is, by definition, the graph $T^{(1)}$ with vertex and edge sets 
$$\text{Vertex}(T^{(1)}) \ = \ \text{Vertex}(T) \sqcup \{\ \{e,\bar{e}\} \ | \ e \in \text{Edge}(T) \}, \ \ \ \ \ \ 
\text{Edge}(T^{(1)}) \ = \  \text{Edge}(T) \times \{ 0,1 \},$$
respectively, where $\overline{(e, \eps)} = (\bar{e}, 1 - \eps)$. \\ 
It is easy to see that each tree $T$ of finite diameter has a unique center, which is a vertex or an edge (this follows essentially from Exercise 3 
on page 21 of \cite{Serre}). Therefore, $T^{(1)}$ has a unique centre, which is a vertex. 
\par 
Let $G$ be a group acting on a tree $T$. Every element of $G$ is either elliptic, i.e. fixes some vertex or some edge of $T$, or is hyperbolic otherwise 
(see, e.g., \cite[I.6]{Serre}). It is clear that there is a canonical way in which $G$ acts on $T^{(1)}$ and on its boundary $\partial T^{(1)}$. 
Note that, if an element of $G$ is elliptic, then it necessary fixes a vertex of $T^{(1)}$. Every hyperbolic element fixes exactly two point the boundary 
$\partial T$ (and also on the boundary $\partial T^{(1)}$), one is attractive and the other one is repulsive. 
By definition, two hyperbolic elements are transverse if they don't have any common fixed points on $\partial T$. 
The action of $G$ on $T$ is said to be of general type if there are two transverse hyperbolic elements.  
\par 
The action of a group $G$ on a set $X$ is called transitive if for every $x, y \in X$, there is an element $g \in G$ satisfying $g x = y$. The action of 
a group $G$ on a topological space $X$ is called minimal if every $G$-orbit is dense in $X$. 
It is obvious that the action of a group amalgamation or an HNN-extension on their corresponding Bass-Serre trees is transitive and, 
therefore, minimal. 
\par 
For an elliptic element $g \in G \setminus \{ 1 \}$, we define a set $\Psi_T(g) \ \equiv \ \partial T^{(1)}_g \cap T^{(1)}_g \ \subset T^{(1)}$. 
Equivalently, $t \in T^{(1)}_g$ belongs to $\Psi_T(g)$ if $t$ has a neighbour that is not fixed by $g$. Clearly, if $g \not= 1$ is elliptic, then 
$\Psi_T(g) \not= \emptyset$. Let $\Upsilon_T(g)$ be the smallest connected subtree of $T^{(1)}$ containing $\Psi_T(g)$. 
We propose the following definitions.
\begin{definition}
Let $G$ act on a tree $T$. We call an elliptic element $g \in G \setminus \{ 1 \}$ \it{finitely fledged} if $\Upsilon_T(g)$ has a finite diameter. 
We call the action of $G$ on $T$ \it{finitely fledged} if every elliptic element  $g \in G \setminus \{ 1 \}$ is finitely fledged. 
\end{definition} 
Note that, in this situation, $\Upsilon_T(g)$ has a unique center in $T^{(2)}$. 
Note also that, if $c \in T^{(2)}$ is the center of $\Upsilon_T(g)$ and if $h \in G$, then $hc$ is the center of $\Upsilon_T(h g h^{-1})$. 
\begin{definition} 
Let $G$ act on a tree $T$. We call an elliptic element $g \in G \setminus \{ 1 \}$ \it{infinitely fledged} if $\Upsilon_T(g)$ has an infinite diameter. 
We call the action of $G$ on $T$ \it{infinitely fledged} if there is an infinitely fledged elliptic element $g \in G \setminus \{ 1 \}$.
\end{definition}
\par 
We will make use of \cite[Proposition 7]{BH86} which, for a group $G$ acting on a tree $T$, states that if the action is of general type and if there is 
an equivariant (with respect to conjugation) map $\delta : G \setminus \{ 1 \} \to T \sqcup \partial T$, then $G$ is not inner amenable. 
Note that, $T \sqcup \partial T$, equipped with the shadow topology, is a compact, totally disconnected topological space 
(see \cite[Section 4.1]{monodshalom}, \cite[Appenix A]{BIO}). 
\par
The following proposition can be regarded as a refinement of \cite[Proposition 0.3]{Stalder}.

\begin{prop} \label{fledged}
Let $G$ be a group acting on a tree $T$, so that the action is of general type. If the action of $G$ on $T$ is finitely fledged, then 
$G$ is not inner amenable. 
\end{prop}

\begin{proof}

Define $X = T^{(2)} \sqcup \partial T^{(2)}$. As in the proof of \cite[Proposition 0.3]{Stalder}, we define a $G$-equivariant map 
$\delta : G \setminus \{1\} \longrightarrow X$ in the following way: \\ 
$\bullet$ If $g \in G$ is hyperbolic, we define $\delta(g)$ to be the attractive point of $g$ on $\partial T^{(2)}$. \\ 
$\bullet$ If $g \in G \setminus \{1\}$ is elliptic, we define $\delta(g)$ to be the center of $\Upsilon_T(g)$ in $T^{(2)}$. 
\par 
The result follows from \cite[Proposition 7]{BH86}.
\end{proof} 

It was proven in \cite{Stalder} that the Baumslag-Solitar group 
$$BS(2,3) \ = \ \{ \tau , b \ | \ \tau^{-1} b^2 \tau = b^3 \}$$
is inner amenable. We note that its action on its Bass-Serre tree $T$ is infinitely fledged. 
\par
To see this, it is easy to verify that the element $b^6$ fixes the linear subtree 
$$T' \ = \ \{ \dots, \tau^{-1} b \tau b \tau^{-1} \langle b \rangle, \tau^{-1} b \tau \langle b \rangle, \tau^{-1} \langle b \rangle, \langle b \rangle, \tau \langle b \rangle, 
\tau b \tau^{-1} \langle b \rangle, \tau b \tau^{-1} b \tau \langle b \rangle, \tau b \tau^{-1} b \tau b \tau^{-1} \langle b \rangle, \dots \}.$$
For example, 
\begin{multline*} 
b^6 \tau b \tau^{-1} b \tau b \tau^{-1} \langle b \rangle = \tau b^9 \cdot b \tau^{-1} b \tau b \tau^{-1} \langle b \rangle = 
\tau b \tau^{-1} b^6 \cdot b \tau b \tau^{-1} \langle b \rangle = \tau b \tau^{-1} b \tau b^9 \cdot b \tau^{-1} \langle b \rangle = \\ 
= \tau b \tau^{-1} b \tau b \tau^{-1} b^6 \langle b \rangle = \tau b \tau^{-1} b \tau b \tau^{-1} \langle b \rangle.
\end{multline*}
Each vertex of $T'$ that starts and ends with the same power of $\tau$, however, has a neighbour that is not fixed by $b^6$. The neighbour 
$\tau b \cdots \tau^{-1} b \tau \cdot \tau \langle b \rangle$ of the vertex $\tau b \cdots \tau^{-1} b \tau \langle b \rangle$ is not fixed by $b^6$, 
and the neighbour $\tau^{-1} b \cdots \tau b \tau^{-1} \cdot \tau^{-1} \langle b \rangle$ of the vertex 
$\tau^{-1} b \cdots \tau b \tau^{-1} \langle b \rangle$ is not fixed by $b^6$. To see this, observe, for example, that 
\begin{multline*} 
b^6 \tau b \tau^{-1} b \tau b \tau^{-1} b \tau \cdot \tau \langle b \rangle = \tau b^9 \cdot b \tau^{-1} b \tau b \tau^{-1} b \tau \cdot \tau \langle b \rangle = 
\tau b \tau^{-1} b^6 \cdot b \tau b \tau^{-1} b \tau \cdot \tau \langle b \rangle = \\ 
\tau b \tau^{-1} b \tau b^9 \cdot b \tau^{-1} b \tau \cdot \tau \langle b \rangle = \tau b \tau^{-1} b \tau b \tau^{-1} b^6 \cdot b \tau \cdot \tau \langle b \rangle = 
\tau b \tau^{-1} b \tau b \tau^{-1} b \tau b^9 \tau \langle b \rangle = \\ 
\tau b \tau^{-1} b \tau b \tau^{-1} b \tau b \cdot b^8 \tau \langle b \rangle = \tau b \tau^{-1} b \tau b \tau^{-1} b \tau b \tau b^{12} \langle b \rangle = 
\tau b \tau^{-1} b \tau b \tau^{-1} b \tau b \tau \langle b \rangle  \not= \tau b \tau^{-1} b \tau b \tau^{-1} b \tau \cdot \tau \langle b \rangle.
\end{multline*}
Therefore, the action of $BS(2,3)$ on $T$ is infinitely fledged.
\par
The proofs of Proposition \ref{fledged} and \cite[Proposition 7]{BH86} imply 

\begin{cor}
Let $G$ be a group acting on a tree $T$, so that the action is of general type. If $m$ is a conjugation invariant mean on $G \setminus \{ 1 \}$, 
then $m$ is supported on the infinitely fledged elliptic elements of $G \setminus \{ 1 \}$. 
\end{cor} 

To conclude this section, we recall that we called a group amenablish if it has no nontrivial $C^*$-simple quotients (\cite[Definition 7.1]{IO}). 
We showed in \cite{IO} that the class on amenablish groups is a radical class, so every group has a unique maximal normal amenablish subgroup, 
the amenablish radical. Also, since the class of amenablish is a radical class, it is closed under extensions. 

\section{Group Amalgamations} 

\subsection{Notations, Definitions, Quasi-Kernels}  

We introduce notations, some of which appear in \cite{IO}:
\par
Let $G = G_0 *_H G_1$ be a nontrivial amalgam. We define sets $T_{j,n} \subset G$ as follows:\\ 
Let $T_{0,0}=T_{1,0}=H$. For $j=0,1$ and $n\geq 1$, let
\[
T_{j,n}=\{g_j\dotsm g_{j+n-1} \ | \  g_i\in G_{i\pmod 2}\setminus H\}.
\]
Now, for $j=0,1$ and $n\geq 0$, let 
\begin{equation}\label{eq:C-jn}
C_{j,n}=\bigcap_{g\in T_{j,n}}gHg^{-1}.
\end{equation}
Next, we consider the quasi-kernels defined in \cite{IO}: 
\begin{equation}\label{eq:K0K1}
K_0=\bigcap_{n\geq 0}C_{0,n}
\quad\text{and}\quad
K_1=\bigcap_{n\geq 0}C_{1,n}. 
\end{equation}
Note that, 
\begin{equation}\label{eq:kerG}
\ker G=K_0 \cap K_1,
\end{equation}
where the kernel of $G$ is defined as 
\[
\ker G = \bigcap_{g\in G}gHg^{-1}.
\]
Also, for $j=0,1,$ $n \geq 1$, and $g_i\in G_{i\pmod 2}\setminus H$, let 
$$K(g_{j+n} \dotsm g_j) = g_{j+n} \dotsm g_j K_j g_j^{-1}\dotsm g_{j+n}^{-1}  \text{    and}$$
$$\bar{K}(g_{j+n} \dotsm g_j) = g_{j+n} \dotsm g_j K_{j-1\pmod2} g_j^{-1}\dotsm g_{j+n}^{-1}.$$
\par 
It follows from \cite[Proposition 3.1]{IO} that $G$ has the unique trace property if and only if $\ker G$ has the unique trace property. 
It also follows from \cite[Theorem 3.9]{BIO} and from Proposition \ref{KjTj} (i) (below) that $G$ is $C^*$-simple if and only if $K_0$ or $K_1$ is trivial or 
non-amenable provided $G$ is a nondegenerate amalgam and $\ker G$ is trivial. 
\par
We need the following results.
 
\begin{lem}\label{K0K1}
$K_0$ and $K_1$ are normal subgroups of $H$. If $\ker G$ is trivial, then $K_0$ and $K_1$ have a trivial intersection and mutually commute.
\end{lem}

\begin{proof}
First statement follows from the observation that for each $h \in H$ and each $j=0,1$,
\[
h \cdot T_{j,n}=\{h \cdot g_j\dotsm g_{j+n-1} \ | \ 	 g_i\in G_{i\pmod 2}\setminus H\} = T_{j,n}.
\]
For the second statement, ($\ref{eq:kerG}$) implies $K_0 \cap K_1 = \ker G = \{1\}$. Take $k_j \in K_j$, $j = 0,1$. Since $K_j \triangleleft H$ for each $j=0,1$, it follows $k_0 k_1^{-1} k_0^{-1} \in K_1$ and $k_1 k_0 k_1^{-1} \in K_0$. Therefore, 
\[ 
K_0 \ni (k_1 k_0 k_1^{-1})k_0^{-1} = k_1(k_0 k_1^{-1} k_0^{-1}) \in K_1,
\]
so $k_1 k_0 k_1^{-1}k_0^{-1} \in K_0 \cap K_1 = \{1\}$.
\end{proof}

\begin{lem}\label{K(...)}
Let $j \in \{0,1\}$ and $g_i\in G_{i\pmod 2}\setminus H$. Then $K(g_{j+n} \dotsm g_j)$ is a proper subgroup of $K_{j+n+1\pmod2}$.
\end{lem}

\begin{proof}
For $m \geq 2$, observe that 
\begin{equation*}
\begin{aligned}
g_j \cdot T_{j,m} \ &= \ \{g_j \cdot \gamma_j \dotsm \gamma_{j+m-1} \ | \ \gamma_i\in G_{i\pmod 2}\setminus H\} \\ 
                    &\supset\  \{g_j \cdot (g_j)^{-1} \gamma_{j+1} \dotsm \gamma_{j+m-1} \ | \ \gamma_i\in G_{i\pmod 2}\setminus H\} \\ 
                    &= \ T_{j+1\pmod2,m-1}.
\end{aligned}
\end{equation*}
It follows by induction on $n \geq 0$ that for $m \geq n+2$, one has 
\begin{equation*}
\begin{aligned}
g_{j+n} \dotsm g_j \cdot T_{j,m} \ &= \ \{g_{j+n} \dotsm g_j \cdot \gamma_j \dotsm \gamma_{j+m-1} \ | \ \gamma_i\in G_{i\pmod 2}\setminus H\} \\ 
                    &\supset\  
\{g_{j+n} \dotsm g_j \cdot (g_j^{-1} \dotsm g_{j+n}^{-1}) \gamma_{j+n+1} \dotsm \gamma_{j+m-1} \ | \ \gamma_i\in G_{i\pmod 2}\setminus H\} \\ 
                    &= \ T_{j+n+1\pmod2,m-n-1}.
\end{aligned}
\end{equation*}
The assertion follows from equations (\ref{eq:C-jn}) and (\ref{eq:K0K1}). 
\end{proof}

\begin{lem}\label{K(.)K(.')}
Let $j \in \{ 0,1 \}$ and $g_i, g_i' \in G_{i\pmod 2}\setminus H$. Then the following hold: \\
(i) If $(g'_j)^{-1} \cdots (g'_{j+n})^{-1} g_{j+n} \cdots g_j \in H$, then $K(g_{j+n} \cdots g_j) = K(g'_{j+n} \cdots g'_j)$. \\ 
(ii) If $\ker G$ is trivial and if $(g'_j)^{-1} \cdots (g'_{j+n})^{-1} g_{j+n} \cdots g_j \notin H$, then $K(g_{j+n} \cdots g_j)$ and $K(g'_{j+n} \cdots g'_j)$ have a trivial intersection and mutually commute. 
\end{lem}

\begin{proof}
Denote $\gamma = (g'_j)^{-1} \cdots (g'_{j+n})^{-1} g_{j+n} \cdots g_j$. \\ 
(i) If $\gamma \in H$, then
\begin{multline*}
(g'_j)^{-1} \cdots (g'_{j+n})^{-1} K(g_{j+n} \cdots g_j) g'_{j+n} \cdots g'_j \\  
= (g'_j)^{-1} \cdots (g'_{j+n})^{-1} g_{j+n} \cdots g_j K_j g_j^{-1} \cdots g_{j+n}^{-1} g'_{j+n} \cdots g'_j \\ 
= \gamma K_j \gamma^{-1} = K_j,
\end{multline*}
where the last equality follows from Lemma \ref{K0K1}. \\
(ii) Note that, if $\gamma \notin H,$ then $\gamma$ starts and ends with elements of $G_j \backslash H$ 
(or itself is an element of $G_j \backslash H$). Therefore, Lemma \ref{K(...)} implies 
\begin{multline*}
(g'_j)^{-1} \cdots (g'_{j+n})^{-1} K(g_{j+n} \cdots g_j) g'_{j+n} \cdots g'_j \\ 
= (g'_j)^{-1} \cdots (g'_{j+n})^{-1} g_{j+n} \cdots g_j K_j g_j^{-1} \cdots g_{j+n}^{-1} g'_{j+n} \cdots g'_j \\ 
= \gamma K_j \gamma^{-1} = K(\gamma) < K_{j+1 \pmod2}.
\end{multline*}
This, combined with 
$$ (g'_j)^{-1} \cdots (g'_{j+n})^{-1} K(g'_{j+n} \cdots g'_j) g'_{j+n} \cdots g'_j = K_j$$
and Lemma \ref{K0K1}, yield the last statement. 
\end{proof}

\begin{remark}\label{Bass-Serre}
Consider the Bass-Serre tree $T = T[G]$ of the group $G = G_0 *_H G_1$, and consider the edge $H$ with ends $G_0$ and $G_1$. 
For each $j=0,1$, denote by $T_j$ the full subtree of $T$ consisting of all vertices $v \in T$ for which $\dist(v, G_j) < \dist(v, G_{j+1 \pmod2})$, and 
denote by $\bar{T}_j$ the full subtree of $T$ consisting of all vertices $v \in T$ for which $\dist(v, G_j) > \dist(v, G_{j+1 \pmod2})$. Note that, 
$T_j = \bar{T}_{j+1 \pmod2}$. Note also that, if $S_i$ are coset representatives for $G_i / H$, then 
\begin{equation} \label{Tj}
T_j = \{ G_j \} \cup \{ s_j \cdots s_{j+n} G_{j+n+1 \pmod2} \ | \ n \geq 0,\ s_k \in S_{k \pmod2} \setminus \{ 1 \} \}.
\end{equation}
\end{remark}

\begin{prop} \label{KjTj}
With the notation of the previous Remark, the following hold: \\ 
(i)  $K_j = G_{(T_j)}$. \\ 
(ii) $K(g_{j+n} \dotsm g_j) = G_{(g_{j+n} \dotsm g_j T)}$ and $\bar{K}(g_{j+n} \dotsm g_j) = G_{(g_{j+n} \dotsm g_j \bar{T})}$.
\end{prop}
\begin{proof}
(i) We have
\begin{equation*} 
\begin{aligned}
h \in K_j \ \ \ & \Longleftrightarrow \\ 
g_{j+n-1}^{-1} \cdots g_j^{-1} h g_j \cdots g_{j+n-1} \in H, \ \ \forall n \geq 0, \ \ \forall g_k \in G_{k \pmod2} \setminus H \ \ \ & \Longleftrightarrow  \\ 
h g_j \cdots g_{j+n-1} \in g_j \cdots g_{j+n-1} H, \ \ \forall n \geq 0, \ \ \forall g_k \in G_{k \pmod2} \setminus H \ \ \ & \Longleftrightarrow  \\ 
h g_j \cdots g_{j+n-1} H = g_j \cdots g_{j+n-1} H, \ \ \forall n \geq 0, \ \ \forall g_k \in G_{k \pmod2} \setminus H \ \ \ & \Longleftrightarrow  \\ 
h \ \text{fixes every edge of} \ T_j \ \ \ & \Longleftrightarrow  \\ 
h \in G_{(T_j)}. & 
\end{aligned}
\end{equation*} 
(ii) As in (i), we have
\begin{equation*} 
\begin{aligned}
h \in K(g_{j+n} \dotsm g_j) \ \ \ & \Longleftrightarrow \\ 
h \in g_{j+n} \dotsm g_j K_j g_j^{-1}\dotsm g_{j+n}^{-1} \ \ \ & \Longleftrightarrow \\ 
g_j^{-1}\dotsm g_{j+n}^{-1} h g_{j+n} \dotsm g_j \in K_j \ \ \ & \Longleftrightarrow \\  
g_j^{-1}\dotsm g_{j+n}^{-1} h g_{j+n} \dotsm g_j \in G_{(T_j)} \ \ \ & \Longleftrightarrow \\  
h  \in g_{j+n} \dotsm g_j G_{(T_j)} g_j^{-1}\dotsm g_{j+n}^{-1} \ \ \ & \Longleftrightarrow \\ 
h  \in  G_{(g_{j+n} \dotsm g_j T_j)}. & 
\end{aligned}
\end{equation*}
$\bar{K}(g_{j+n} \dotsm g_j) = G_{(g_{j+n} \dotsm g_j \bar{T})}$ follows in a similar way.
\end{proof}
\par
Now, chose representatives $S_j$ of $G_j/H$ containing $\{ 1 \}$ for $j=0,1$, and denote $S'_j = S_j \setminus \{ 1 \}$. 
\par
Assume that $\ker G = \{ 1 \}$. 
\par 
Then, from Lemmas \ref{K0K1} and \ref{K(...)}, it follows that $K(s_j) \cap K_j = \{ 1 \}$ and that 
$K(s_j)$ and $K_j$ mutually commute for $s_j \in S'_j$ and for $j=0,1$. Also, from Lemma \ref{K(.)K(.')}, it follows that if $s_j, t_j \in S'_j$ 
are different, then $K(s_j)$ and $K(t_j)$ have a trivial intersection and mutually commute. Likewise, it follows from Lemma \ref{K(.)K(.')} that for 
$n \geq 1$ and for $s_i, t_i \in S'_{i \pmod2}$, $K(s_{j+n} s_{j+n-1} \cdots s_j) = K(t_{j+n} t_{j+n-1} \cdots t_j)$ if and only if 
$(s_{j+n}, \dots, s_j) = (t_{j+n}, \dots, t_j)$. If $(s_{j+n}, \dots, s_j) \not= (t_{j+n}, \dots, t_j)$, then 
$K(s_{j+n} s_{j+n-1} \cdots s_j)$ and $K(t_{j+n} t_{j+n-1} \cdots t_j)$ have a trivial intersection and mutually commute. Moreover, from Lemmas 
\ref{K0K1} and \ref{K(...)}, it follows that $K(s_{j+n+1} s_{j+n} \cdots s_j)$ and $K(t_{j+n} t_{j+n-1} \cdots t_j)$ have a trivial intersection and mutually 
commute for any choice of $s_i$ and $t_i.$
\par
Thus, if we consider for $j=0,1$ and for $n \in \mathbb{N}_0$, the following subgroups of $H$,
\begin{equation} \label{calK'}
{\mathcal K}'(j,n) = \underset{s_i \in S'_{i\pmod 2}}{\bigoplus} K(s_j \dotsm s_{j+n}) \text{  and}
\end{equation}
\begin{equation} \label{calK}
{\mathcal K}(j,n) = \underset{s_i \in S'_{i\pmod 2}}{\underset{\gamma_j \in S_j}{\bigoplus}} K(\gamma_j s_{j+1} \dotsm s_{j+n}),
\end{equation}
then it is easy to see that:
\begin{prop} 
${\mathcal K}(j,n)$ is a normal subgroup of $G_j$, and ${\mathcal K}'(j,n)$ is a normal subgroup of $H$.
\end{prop}

\begin{remark}
From Lemma \ref{K(...)}, it follows that there are subgroups of $K_1$ isomorphic to $K_0$ and vice versa. Consequently, $K_0 = \{ 1 \}$ if and only if 
$K_1 = \{ 1 \}$. In this situation, the groups ${\mathcal K}(j,n)$ and ${\mathcal K}'(j,n)$ are all trivial. 
\end{remark}

\subsection{A Family of Examples}

For $j=0,1$, consider nonempty sets $I'_j$, and let $I_j = I'_j \sqcup \{ \iota_j \}$ and transitive permutation groups $\Gamma_j$ on $I_j$ 
having stabilizer groups $\Gamma'_j \ \equiv \ (\Gamma_j)_{\iota_j}$ that are not both trivial. We define a family of groups that depend on 
$\Gamma_0$ and $\Gamma_1$ as follows: 
$$G[\Gamma_0,\Gamma_1] \ \equiv \ G[I_0,I_1;\iota_0,\iota_1;\Gamma_0,\Gamma_1] \ \equiv \ G_0 *_H G_1,$$
where $H \ = \langle Q_0, Q_1 \rangle$ and where 
\begin{multline*}
Q_j = \\ 
\langle \  \{ \  h_j( i_j, i_{j+1}, \dots, i_{j+n} ; \sigma_{j+n+1}) \ 
| \ n \in \mathbb{Z}_0, \ i_k \in I'_{k \pmod 2},\ \sigma_{j+n+1} \in \Gamma'_{j+n+1\pmod2} \ \} \\ 
\ \cup \ \{ \ g_j( \sigma_j) \ | \ \sigma_j \in \Gamma_j'\ \} \ \rangle 
\end{multline*}
for $j=0,1$. Finally, $G_j = \langle H \cup \{ \ g_j(\sigma_j) \ | \ \sigma_j \in \Gamma_j \setminus \Gamma'_j \ \} \rangle$.
\par
The following relations hold (there are redundancies): \\ 
(R1) The groups $Q_0$ and $Q_1$ mutually commute. \\ 
(R2) For $j=0,1$, $0 \leq m \leq n$, $i_k, s_k \in I'_{k \pmod 2}$ with $(i_j, \dots, i_{j+m}) \not= (s_j,  \dots, s_{j+m})$, 
and for $\sigma_k \in \Gamma'_{k\pmod2}$, the elements
$$h_j( s_j,  \dots, s_{j+m} ; \sigma_{j+m+1}) \ \ \text{   and   } \ \ h_j( i_j, \dots, i_{j+m}, \dots, i_{j+n} ; \sigma_{j+n+1})$$ 
commute. \\ 
(R3) For $j=0,1$, $0 \leq m < n$, $i_k \in I'_{k \pmod 2}$, and for $\sigma_k \in \Gamma'_{k\pmod2}$, the following holds 
\begin{multline*}
h_j( i_j,  \dots, i_{j+m} ; \sigma_{j+m+1}) h_j( i_j, \dots, i_{j+m}, i_{j+m+1} \dots, i_{j+n} ; \sigma_{j+n+1}) 
h_j( i_j,  \dots, i_{j+m} ; \sigma_{j+m+1})^{-1} \\  
= h_j( i_j, \dots, i_{j+m}, \sigma_{j+m+1}(i_{j+m+1}), \dots, i_{j+n} ; \sigma_{j+n+1}). 
\end{multline*}
(R4) For $j=0,1$, $m \in \mathbb{Z}_0$, $i_k \in I'_{k \pmod 2}$, and $\sigma_{j+m+1}, \tilde{\sigma}_{j+m+1} \in \Gamma'_{j+m+1\pmod2}$, 
the following hold
$$ h_j( i_j, \dots, i_{j+m} ; \id) = 1,$$
$$ h_j( i_j, \dots, i_{j+m} ; \tilde{\sigma}_{j+m+1}) h_j( i_j, \dots, i_{j+m} ; \sigma_{j+m+1}) \ = \ 
h_j( i_j, \dots, i_{j+m} ; \tilde{\sigma}_{j+m+1} \sigma_{j+m+1}),$$
and 
$$ h_j( i_j, \dots, i_{j+m} ; \sigma_{j+m+1})^{-1} = h_j( i_j, \dots, i_{j+m} ; \sigma_{j+m+1}^{-1}).$$
(R5) For $j=0,1$ and $\sigma_j, \tilde{\sigma}_j \in \Gamma_j$, the following hold
$$ g_j(\id) = 1,\ \ \ \ \ g_j(\sigma_j) g_j(\tilde{\sigma}_j) = g_j(\sigma_j \tilde{\sigma}_j), \ \ \ \text{ and }\ \ \ g_j(\sigma_j)^{-1} = g_j(\sigma_j^{-1}).$$
(R6) For $j=0,1$,\ $m \in \mathbb{Z}_0$, $i_k \in I'_{k\pmod2}$, $\sigma_j \in \Gamma_j$, and $\sigma'_{j+m+1} \in \Gamma'_{j+m+1\pmod2}$, 
the following holds
\begin{multline*}
g_j(\sigma_j)  h_j(i_j, \dots, i_{j+m}; \sigma'_{j+m+1})  g_j(\sigma_j)^{-1} \ = \\
\begin{cases}
g_{j+1\pmod2}(\sigma'_{j+m+1}) \text{  if } \sigma_j(i_j) = \iota_j \text{ and } \ m = 0, \\ 
h_{j+1\pmod2}(i_{j+1}, \dots, i_{j+m}; \sigma'_{j+m+1}) \text{  if } \sigma_j(i_j) = \iota_j \text{ and } \ m \geq 1, \\
h_j(\sigma_j(i_j), \dots, i_{j+m}; \sigma'_{j+m+1}) \text{  if } \sigma_j(i_j) \not= \iota_j, \text{ and } \ m \geq 0. 
\end{cases}
\end{multline*}

\subsection{Some Basic Properties of the Examples and Their Quasi-Kernels}

For a group $G[I_0,I_1;\iota_0,\iota_1;\Gamma_0,\Gamma_1]$, let's note that $\Index[G_j : H] = \#(I_j),\ j =0,1$. To see this, recall that $\Gamma_j$ acts transitively on $I_j$, and for $i \in I'_j$, let $\tau_j^i \in \Gamma_j$ be such that $\tau_j^i(\iota_j) = i$. 
Let's denote $\gamma_j^i \equiv g_j(\tau_j^i)$, and take an element $\sigma \in \Gamma_j \setminus \Gamma_j'$ with $\sigma(\iota_j) = i$. 
Then $(\tau_j^i)^{-1} \circ \sigma (\iota_j) = \iota_j$, so 
$g_j((\tau_j^i)^{-1} \circ \sigma (\iota_j)) \in H$, and therefore $g_j(\sigma) \in g_j(\tau_j^i) H = \gamma_j^i H$. Thus
\begin{equation} \label{classes} 
G_j \ = \ H \sqcup \underset{i \in I'_j}{\bigsqcup} \gamma_j^i H.
\end{equation}

Consider the canonical action of $G = G[I_0,I_1;\iota_0,\iota_1;\Gamma_0,\Gamma_1]$ on its Bass-Serre tree $T = T[G]$. Adjacent vertices to the vertex 
$G_j$ different from $G_{j+1 \pmod2}$ can be indexed by the set $I'_j$, so we denote the vertex $G_j$ by $v(\iota_j)$, and for $i \in I'_j$, we denote 
the vertex $\gamma_j^i G_{j+1 \pmod2}$ by $v(\iota_j,i)$. Also, for $i_t \in I'_{t \pmod2}$, we denote the vertex 
$\gamma_j^{i_j} \gamma_{j+1 \pmod2}^{i_{j+1}} \cdots \gamma_{j+k \pmod2}^{i_{j+k}} G_{j+k+1 \pmod 2}$ by 
$v(\iota_j, i_j, \dots, i_{j+k})$. Note that, using notation from Remark \ref{Bass-Serre}, $T_j$ is the full subtree of $T$ containing the vertex 
$v(\iota_j)$ and the vertices $v(\iota_j, i_j, \dots, i_{j+k})$, where $i_t \in I'_{t \pmod2}$ and $j = 0,1$.

\begin{remark} \label{nondegenerate}
Notice that in the case $\#(I_0) = \#(I_1) = 2$ is impossible because of the requirement for non-triviality of the stabilizers, so 
$\#(I_0) \geq 3$ or $\#(I_1) \geq 3$. Therefore, the corresponding Bass-Serre tree is not a linear tree and the amalgam is nondegenerate 
(see \cite[Proposition 19]{Harpepreaux}).
\end{remark}

\begin{remark}
There is a resemblance of our examples (as well as the HNN-extension examples below) with the groups introduced by Le Boudec in \cite{leboudec16}. 
In fact, the example from \cite[Section 4]{IO} is isomorphic to one of the groups from \cite{leboudec16} (see \cite[Remark 4.7]{IO}). 
It can be shown that all of our examples satisfy the conditions of \cite[Theorem A]{leboudec17}. 
One benefit is that our groups are given concretely in terms of generators and relations. Whenever $\#(I_0) \not= \#(I_1)$, none of the groups 
$G[I_0,I_1;\iota_0,\iota_1;\Gamma_0,\Gamma_1]$ is covered in \cite{leboudec16} or in \cite[Theorem C]{leboudec17}. 
\end{remark} 

\begin{remark}
It is immediate from \cite[Theorem VI.9]{baumslag} that our examples are not finitely presented since $H$ is never finitely generated.
\end{remark}
\par
We need some easy facts about $G = G[I_0,I_1;\iota_0,\iota_1;\Gamma_0,\Gamma_1]$. 

\begin{lem} \label{generation}
(i) Let $m \geq 0$, $i_t \in I'_{t \pmod2}$, and $\sigma'_{j+m+1} \in \Gamma'_{j+m+1\pmod2}$. Then, 
\begin{multline*} 
h_j(i_j, \dots, i_{j+m}; \sigma'_{j+m+1}) = \\ 
= \gamma_j^{i_j} \cdots \gamma_{j+m \pmod2}^{i_{j+m}}
                                               g_{j+m+1 \pmod2}(\sigma'_{j+m+1}) (\gamma_{j+m \pmod2}^{i_{j+m}})^{-1} \cdots (\gamma_j^{i_j})^{-1}.
\end{multline*}

(ii) Every element $h \in Q_j$ can be written as 
$$h = g_j (\sigma'_j) \prod_{k=1}^{m} h_j(i_j^k, \dots, i_{j+n_k}^k; \sigma''_k),$$
where $m \geq 1, \ \sigma'_j \in \Gamma'_j,\ \sigma''_k \in \Gamma'_{j+n_k+1 \pmod2},\ 0 \leq n_1 \leq \dots \leq n_m$, and 
$i_t^k \in I'_{t \pmod2}$. 
\par
(iii) Every element $g \in T_{j,n}$ can be written as 
$$g = \gamma_j^{k_j} \cdots \gamma_{j+n-1\pmod2}^{k_{j+n-1}} h,$$
where $h \in H$ and $k_t \in I'_{t \pmod2}$.
\end{lem}
\begin{proof}
(i) For $\sigma_j \in \Gamma_j$, with $\sigma_j(\iota_j) = i_j$, (R6), read backwards, gives
$$h_j(i_j, \dots, i_{j+m}; \sigma'_{j+m+1}) = g_j(\sigma_j) h_{j+1\pmod2}(i_{j+1}, \dots, i_{j+m}; \sigma'_{j+m+1}) g_j(\sigma_j)^{-1}.$$
Now, (i) follows by induction. \\ 
(ii) Follows from (R4) and (R5) applied several times. \\ 
(iii) Follows from equation (\ref{classes}) and the structure of the amalgams. 
\end{proof}
\begin{lem} \label{action}
Let $m \geq 0$, $i_t \in I'_{t \pmod2}$, $\sigma'_j \in \Gamma'_j$, and $\sigma'_{j+m+1} \in \Gamma'_{j+m+1\pmod2}$. Then: 
\par
(i) For $n > m$, the following holds 
\begin{multline*}
h_j(i_j, \dots, i_{j+m}; \sigma'_{j+m+1}) v(\iota_j, i_j, \dots,i_{j+m}, i_{j+m+1}, i_{j+m+2}, \dots, i_{j+n}) = \\ 
= v(\iota_j, i_j, \dots,i_{j+m}, \sigma'_{j+m+1}(i_{j+m+1}), i_{j+m+2}, \dots, i_{j+n});
\end{multline*}
Also, $g_j(\sigma'_j) v(\iota_j, k_j) = v(\iota_j, \sigma'_j(k_j))$ for $k_j \in I'_j$. 
\par 
(ii) For $\sigma'_j \in \Gamma'_j$, it follows $g_j(\sigma'_j) \in G_{v(\iota_j)}$ and 
$h_j(i_j, \dots, i_{j+m}; \sigma'_{j+m+1}) \in G_{v(\iota_j, i_j, \dots, i_{j+m})}$. 
\par
(iii) For $\sigma'_j \in \Gamma'_j$, it follows $g_j(\sigma'_j) \in K_{j+1 \pmod2}$. 
\par
(iv) For $\sigma'_{j+m+1} \in \Gamma_{j+m+1 \pmod 2}$, it follows $h_j(i_j, \dots, i_{j+m}; \sigma'_{j+m+1}) \in K_{j+1 \pmod2}$. 
\par
(v) For $n \geq m$ and for $s_t \in I'_{t \pmod 2}$, with $(i_j, \dots, i_{j+m}) \not= (s_j,  \dots, s_{j+m})$, it follows 
$$h_j(s_j, \dots, s_{j+m}; \sigma'_{j+m+1}) \in G_{v(\iota_j, i_j, \dots, i_{j+m}, \dots, i_{j+n})}.$$
\end{lem}

\begin{proof} 
(i) First note that, by the structure of the amalgams and from Lemma \ref{generation} (i) and (iii), it follows that there are $t_l \in I'_l$ and a 
$\chi \in H$ satisfying 
$$(\gamma_{j+m+2 \pmod2}^{i_{j+m+2}} \cdots \gamma_{j+n \pmod2}^{i_{j+n}})^{-1} = 
\chi \gamma_{j+n \pmod2}^{t_{j+n}} \cdots \gamma_{j+m+2 \pmod2}^{t_{j+m+2}}.$$ 
Then note that, 
$$\sigma \equiv (\tau^{\sigma'_{j+m+1} (i_{j+m+1})}_{j+m+1})^{-1} \circ \sigma'_{j+m+1} \circ \tau_{j+m+1}^{i_{j+m+1}} 
\in \Gamma'_{j+m+1 \pmod2},$$
since $\sigma$ fixes $\iota_{j+m+1 \pmod2}$. Then, it follows that 
\begin{multline*} 
(\gamma_{j+m+2 \pmod2}^{i_{j+m+2}} \cdots \gamma_{j+n \pmod2}^{i_{j+n}})^{-1}  g_{j_m+1}(\sigma) 
\gamma_{j+m+2 \pmod2}^{i_{j+m+2}} \cdots \gamma_{j+n \pmod2}^{i_{j+n}} = \\ 
\chi h_{j+n}(t_{j+n}, \dots, t_{j+m+2}; \sigma) \chi^{-1}.
\end{multline*}
Therefore,  
\begin{multline*}
h_j(i_j, \dots, i_{j+m}; \sigma'_{j+m+1}) v(\iota_j, i_j, \dots,i_{j+m}, i_{j+m+1}, i_{j+m+2}, \dots, i_{j+n}) = \\ 
= \gamma_j^{i_j} \cdots \gamma_{j+m \pmod2}^{i_{j+m}} g_{j+m+1 \pmod2}(\sigma'_{j+m+1})
 (\gamma_{j+m \pmod2}^{i_{j+m}})^{-1} \cdots (\gamma_j^{i_j})^{-1} \cdot \\ 
\cdot \gamma_j^{i_j} \cdots \gamma_{j+m \pmod2}^{i_{j+m}} \gamma_{j+m+1 \pmod2}^{i_{j+m+1}} 
\gamma_{j+m+2 \pmod2}^{i_{j+m+2}} \cdots \gamma_{j+n \pmod2}^{i_{j+n}} G_{j+n+1 \pmod 2} = \\ 
= \gamma_j^{i_j} \gamma_{j+1 \pmod2}^{i_{j+1}} \cdots \gamma_{j+m \pmod2}^{i_{j+m}} 
g_{j+m+1 \pmod2}(\sigma'_{j+m+1}) \gamma_{j+m+1 \pmod2}^{i_{j+m+1}} \cdot \\ 
\cdot \gamma_{j+m+2 \pmod2}^{i_{j+m+2}} \cdots \gamma_{j+n \pmod2}^{i_{j+n}} G_{j+n+1 \pmod 2} = \\ 
= \gamma_j^{i_j} \gamma_{j+1 \pmod2}^{i_{j+1}} \cdots \gamma_{j+m \pmod2}^{i_{j+m}} 
\gamma^{\sigma'_{j+m+1} (i_{j+m+1})}_{j+m+1} g_{j+m+1 \pmod2}(\sigma) \cdot \\ 
\cdot \gamma_{j+m+2 \pmod2}^{i_{j+m+2}} \cdots \gamma_{j+n \pmod2}^{i_{j+n}} G_{j+n+1 \pmod 2} = \\ 
= \gamma_j^{i_j} \gamma_{j+1 \pmod2}^{i_{j+1}} \cdots \gamma_{j+m \pmod2}^{i_{j+m}} 
\gamma^{\sigma'_{j+m+1} (i_{j+m+1})}_{j+m+1} \gamma_{j+m+2 \pmod2}^{i_{j+m+2}} \cdots \gamma_{j+n \pmod2}^{i_{j+n}} \cdot \\ 
(\gamma_{j+m+2 \pmod2}^{i_{j+m+2}} \cdots \gamma_{j+n \pmod2}^{i_{j+n}})^{-1}
g_{j+m+1 \pmod2}(\sigma) \gamma_{j+m+2 \pmod2}^{i_{j+m+2}} \cdots \gamma_{j+n \pmod2}^{i_{j+n}} \cdot \\ 
G_{j+n+1 \pmod 2} = \\ 
= \gamma_j^{i_j} \gamma_{j+1 \pmod2}^{i_{j+1}} \cdots \gamma_{j+m \pmod2}^{i_{j+m}} 
\gamma^{\sigma'_{j+m+1} (i_{j+m+1})}_{j+m+1} \gamma_{j+m+2 \pmod2}^{i_{j+m+2}} \cdots \gamma_{j+n \pmod2}^{i_{j+n}} \cdot \\ 
\chi h_{j+n}(t_{j+n}, \dots, t_{j+m+2}; \sigma) \chi^{-1} G_{j+n+1 \pmod 2} = \\ 
= v(\iota_j, i_j, \dots,i_{j+m}, \sigma'_{j+m+1}(i_{j+m+1}), i_{j+m+2}, \dots, i_{j+n})
\end{multline*}
Second statement is clear. \\ 
(ii) First claim is obvious. Second claim follows from 
\begin{multline*} 
h_j(i_j, \dots, i_{j+m}; \sigma'_{j+m+1}) v(\iota_j, i_j, \dots, i_{j+m}) = \\ 
\gamma_j^{i_j} \cdots \gamma_{j+m \pmod2}^{i_{j+m}} g_{j+m+1 \pmod2}(\sigma'_{j+m+1})
 (\gamma_{j+m \pmod2}^{i_{j+m}})^{-1} \cdots (\gamma_j^{i_j})^{-1} \cdot \\ 
\gamma_j^{i_j} \gamma_{j+1 \pmod2}^{i_{j+1}} \cdots \gamma_{j+m \pmod2}^{i_{j+m}} G_{j+m+1 \pmod 2} = \\ 
\gamma_j^{i_j} \cdots \gamma_{j+m \pmod2}^{i_{j+m}} g_{j+m+1 \pmod2}(\sigma'_{j+m+1}) G_{j+m+1 \pmod 2} = \\ 
v(\iota_j, i_j, \dots, i_{j+m}).
\end{multline*} 
(iii) It follows by Proposition \ref{KjTj} (i) that this is equivalent to $g_j(\sigma'_j) \in G_{(T_{j+1 \pmod2})}$. 
From $g_j(\sigma'_j) \in H$, it immediately follows that $g_j(\sigma'_j) \in G_{v(\iota_{j+1 \pmod2})} = G_{j+1 \pmod2}$. 
It remains to show that, for any $n \geq 1$ and any $i_t \in I'_t$, it follows 
$g_j(\sigma'_j) \in G_{v(\iota_{j+1 \pmod2}, i_{j+1}, \dots, i_{j+n})}$. 
From the argument at the beginning of the proof of point (i), it follows that 
$$(\gamma_{j+n \pmod2}^{i_{j+n}})^{-1} \cdots (\gamma_{j+1 \pmod2}^{i_{j+1}})^{-1} = 
\chi' \gamma_{j+n \pmod2}^{k_{j+n}} \cdots \gamma_{j+1 \pmod2}^{k_{j+1}}$$
for some $k_l \in I'_l$ and some $\chi' \in H$. Consequently, 
\begin{multline*}
(\gamma_{j+n \pmod2}^{i_{j+n}})^{-1} \cdots (\gamma_{j+1 \pmod2}^{i_{j+1}})^{-1} 
g_j(\sigma'_j) \gamma_{j+1 \pmod2}^{i_{j+1}} \cdots \gamma_{j+n \pmod2}^{i_{j+n}} = \\ 
= \chi h_{j+n \pmod2}(\iota_{j+n \pmod2}, k_{j+n \pmod2}, \dots, k_{j+1 \pmod2}; \sigma'_j) \chi^{-1}.
\end{multline*}
Finally, 
\begin{multline*}
g_j(\sigma'_j) v(\iota_{j+1 \pmod2}, i_{j+1}, \dots, i_{j+n}) = \\ 
= g_j(\sigma'_j) \gamma_{j+1 \pmod2}^{i_{j+1}} \cdots \gamma_{j+n \pmod2}^{i_{j+n}} G_{j+n+1 \pmod2} = \\ 
= \gamma_{j+1 \pmod2}^{i_{j+1}} \cdots \gamma_{j+n \pmod2}^{i_{j+n}} 
\cdot (\gamma_{j+n \pmod2}^{i_{j+n}})^{-1} \cdots (\gamma_{j+1 \pmod2}^{i_{j+1}})^{-1} \cdot \\ 
\cdot g_j(\sigma'_j) \gamma_{j+1 \pmod2}^{i_{j+1}} \cdots \gamma_{j+n \pmod2}^{i_{j+n}} G_{j+n+1 \pmod2} = \\ 
= \gamma_{j+1 \pmod2}^{i_{j+1}} \cdots \gamma_{j+n \pmod2}^{i_{j+n}} \cdot \\ 
\cdot \chi h_{j+n \pmod2}(\iota_{j+n \pmod2}, k_{j+n \pmod2}, \dots, k_{j+1 \pmod2}; \sigma'_j) \chi^{-1} G_{j+n+1 \pmod2} = \\ 
= \gamma_{j+1 \pmod2}^{i_{j+1}} \cdots \gamma_{j+n \pmod2}^{i_{j+n}} G_{j+n+1 \pmod2} = \\ 
= v(\iota_{j+1 \pmod2}, i_{j+1}, \dots, i_{j+n}).
\end{multline*}
(iv) From Lemma \ref{generation} (i), we write 
\begin{multline*} 
h_j(i_j, \dots, i_{j+m}; \sigma'_{j+m+1}) = \\ 
= \gamma_j^{i_j} \cdots \gamma_{j+m \pmod2}^{i_{j+m}}
                                               g_{j+m+1 \pmod2}(\sigma'_{j+m+1}) (\gamma_{j+m \pmod2}^{i_{j+m}})^{-1} \cdots (\gamma_j^{i_j})^{-1}.
\end{multline*}
By (iii), we have $g_{j+m+1 \pmod2}(\sigma'_{j+m+1}) \in K_{j+m \pmod2}$, and therefore, by definition and by Lemma \ref{K(...)}, it follows 
$h_j(i_j, \dots, i_{j+m}; \sigma'_{j+m+1}) \in K(\gamma_j^{i_j} \cdots \gamma_{j+m \pmod2}^{i_{j+m}}) < K_{j+1 \pmod2}$. \\ 
(v) Note that, in $\gamma \equiv (\gamma_{j+m \pmod2}^{s_{j+m}})^{-1} \cdots (\gamma_j^{s_j})^{-1} 
\gamma_j^{i_j}  \cdots \gamma_{j+n \pmod2}^{i_{j+n}}$ there is an even number $0 \geq 2r < 2m+2$ of cancelations. Therefore, 
$\gamma \in T_{j+m \pmod2, m+n+2-2r-1}=T_{j+m \pmod2, m+n+1-2r}$ starts with an element of $G_{j+m \pmod2}$ and ends with an element of 
$G_{j+n \pmod2}$. By Lemma \ref{generation} (iii), it can be written as 
$$\gamma = \gamma_{j+m \pmod2}^{k_{j+m}} \cdots \gamma_{j+2m+n-2r\pmod2}^{k_{j+2m+n-2r}} h = 
\gamma_{j+m \pmod2}^{k_{j+m}} \cdots \gamma_{j+n\pmod2}^{k_{j+2m+n-2r}} h$$
for some $k_t \in I'_t$. Then, 
\begin{equation*} 
\begin{aligned}
h_j(s_j, \dots, s_{j+m}; \sigma'_{j+m+1}) \in G_{v(\iota_j, i_j, \dots, i_{j+m}, \dots, i_{j+n})} \ \ \ & \Longleftrightarrow \\ 
\gamma_j^{s_j} \cdots \gamma_{j+m \pmod2}^{s_{j+m}} g_{j+m+1 \pmod2}(\sigma'_{j+m+1}) 
(\gamma_{j+m \pmod2}^{s_{j+m}})^{-1} \cdots (\gamma_j^{s_j})^{-1} \in \\ 
\in G_{v(\iota_j, i_j, \dots, i_{j+m}, \dots, i_{j+n})} \ \ \ & \Longleftrightarrow \\ 
g_{j+m+1 \pmod2}(\sigma'_{j+m+1})  \in \\  
\in (\gamma_{j+m \pmod2}^{s_{j+m}})^{-1} \cdots (\gamma_j^{s_j})^{-1} G_{v(\iota_j, i_j, \dots, i_{j+m}, \dots, i_{j+n})}
\gamma_j^{s_j} \cdots \gamma_{j+m \pmod2}^{s_{j+m}}  \ \ \ & \Longleftrightarrow \\  
g_{j+m+1 \pmod2}(\sigma'_{j+m+1})  \in
G_{(\gamma_{j+m \pmod2}^{s_{j+m}})^{-1} \cdots (\gamma_j^{s_j})^{-1} v(\iota_j, i_j, \dots, i_{j+m}, \dots, i_{j+n})} \ \ \ & \Longleftrightarrow \\  
g_{j+m+1 \pmod2}(\sigma'_{j+m+1})  \in 
G_{(\gamma_{j+m \pmod2}^{s_{j+m}})^{-1} \cdots (\gamma_j^{s_j})^{-1} 
\gamma_j^{i_j}  \cdots \gamma_{j+n \pmod2}^{i_{j+n}} G_{j+n+1 \pmod 2}} \ \ \ & \Longleftrightarrow \\ 
g_{j+m+1 \pmod2}(\sigma'_{j+m+1})  \in  
G_{\gamma_{j+m \pmod2}^{k_{j+m}} \cdots \gamma_{j+n\pmod2}^{k_{j+2m+n-2r}} h G_{j+n+1 \pmod 2}}  \ \ \ & \Longleftrightarrow \\
g_{j+m+1 \pmod2}(\sigma'_{j+m+1})  \in  
G_{v(\iota_{j+m \pmod2}, k_{j+m} \dots, k_{j+2m+n-2r})}. & 
\end{aligned}
\end{equation*}
Last line holds according to (iii).
\end{proof}

\begin{prop} \label{KjQj}
For a group $G = G[I_0,I_1;\iota_0,\iota_1;\Gamma_0,\Gamma_1]$ the following hold: \\ 
(i) $K_j = Q_{j+1 \pmod2}$. \\ 
(ii) $\ker G = \{ 1 \}$.
\end{prop}

\begin{proof}
(i) It follows by Lemma \ref{action} (iii) and (iv) that $Q_j < K_{j+1 \pmod2}$ and by Lemma \ref{K0K1} that $K_{j+1 \pmod2} \vartriangleleft H$. 
Therefore 
$$Q_j < K_{j+1 \pmod2} \vartriangleleft H.$$
Since $Q_0$ and $Q_1$ commute and together generate $H$, it is enough to show that if $h \in Q_j$, then $h \notin K_j$. 
\par
For $\sigma'_j \in \Gamma'_j \setminus \{ 1 \}$, let $\rho, \kappa \in I'_j$ be such that $\rho \not= \kappa$ and $\sigma'_j(\rho) = \kappa$. Then, 
$g_j(\sigma'_j) v(\iota_j, \rho) = v(\iota_j, \kappa)$ by Lemma \ref{action} (i). Therefore, $g_j(\sigma'_j) \notin G_{(T_j)} = K_j$. 
\par
Take an element $h \in Q_j$, and assume $h \not= g_j(\sigma'_j)$. Then, by Lemma \ref{generation} (ii)  and (R2), it can be written as 
$$h = g_j(\sigma'_j) \cdot h_1 \cdot h_2 \cdot h_j(s_j, \dots, s_{j+n}; \omega'),$$
where 
$$h_1 = \prod_{k=1}^{m} h_j(i_j^k, \dots, i_{j+n_k}^k; \theta'_k)  \text{  and  }  h_2 = \prod_{l=m+1}^{r} h_j(i_j^l, \dots, i_{j+n}^l; \xi'_l).$$ 
In the above expression, 
$r \geq m \geq 0, \ \sigma'_j \in \Gamma'_j,\ \theta'_k \in \Gamma'_{j+n_k+1 \pmod2},\ \omega', \xi'_l \in \Gamma'_{j+n+1 \pmod2},\ 
0 \leq n_1 \leq \dots \leq n_m < n, \ i_t^p \in I'_{t \pmod2}$, and 
$(i_j^l, \dots, i_{j+n}^l) \not= (s_j, \dots, s_{j+n}), \ \forall l \in \{ m+1, \dots, r \}$.  
\par 
Since $\omega' \in \Gamma'_{j+n+1 \pmod2}$ is nontrivial, there exist $\rho, \kappa \in I'_{j+n+1 \pmod2}$, with 
$\rho \not= \kappa$, satisfying $\omega'(\rho) = \kappa$. It follows from Lemma \ref{action} (i) that 
$$h_j(s_j, \dots, s_{j+n}; \omega') v(\iota_j, s_j, \dots, s_{j+n}, \rho) = v(\iota_j, s_j, \dots, s_{j+n}, \kappa).$$
It follows from Lemma \ref{action} (v) that 
$$h_2 v(\iota_j, s_j, \dots, s_{j+n}, \kappa) = v(\iota_j, s_j, \dots, s_{j+n}, \kappa).$$ 
It follows again from Lemma \ref{action} (i) that 
$$g_j(\sigma'_j) h_1 v(\iota_j, s_j, \dots, s_{j+n}, \kappa) = v(\iota_j, p_j \dots, p_{j+n}, \kappa)$$
for some $p_t \in I'_t$. 
\par
We conclude that 
$$h v(\iota_j, s_j, \dots, s_{j+n}, \rho) = v(\iota_j, p_j \dots, p_{j+n}, \kappa) \ \not= \ v(\iota_j, s_j, \dots, s_{j+n}, \rho),$$
and therefore $h \notin G_{(T_j)} = K_j$. \\ 
(ii) We have by the definition of the group $G = G[I_0,I_1;\iota_0,\iota_1;\Gamma_0,\Gamma_1]$ that $Q_1 \cap Q_0 = \{ 1 \}$. Therefore, 
$\ker G = K_0 \cap K_1 = Q_1 \cap Q_0 = \{ 1 \}$.
\end{proof}

We turn to the structure of the groups $K_j = Q_{j+1 \pmod 2}$ for $j = 0,1$. From the construction of 
$G = G[I_0,I_1;\iota_0,\iota_1;\Gamma_0,\Gamma_1]$, it is clear that the following holds:
\begin{multline*}
{\mathcal K}'(j,n) = \\ 
\langle \  \{ \  h_j( i_j, i_{j+1}, \dots, i_{j+n} ; \sigma_{j+n+1}) \ 
| \ n \in \mathbb{Z}_0, \ i_k \in I'_{k \pmod 2},\ \sigma_{j+n+1} \in \Gamma'_{j+n+1\pmod2} \ \} \ \rangle.
\end{multline*}
Therefore, there is a representation 
\begin{equation} \label{recursiveK} 
Q_{j+1\pmod2} \cong {\mathcal K}'(j,1) \rtimes \Gamma'_j,
\end{equation}
with the obvious action. It can be written also "recursively" as the wreath product 
\begin{equation} \label{recursive} 
Q_{j+1\pmod2} \cong Q_j \wr_{I_j'} \Gamma'_j,
\end{equation}
where $\Gamma_j'$ acts on $\#(I'_j)$ coppies of $K_{j+1\pmod2} \cong Q_j$ in the obvious way. 
\par 
We want to express $Q_j$ as a direct limit of wreath products. Let's denote 
$$Q_j(0) \ = \ \{ \ g_j(\sigma'_j) \ | \ \sigma'_j \in \Gamma'_j \ \} \ \cong \ \Gamma'_j.$$
For $n \geq 1$, we denote 
\begin{multline*} 
Q_j(n) \ = \ \langle \ \{ \  h_j( i_j, i_{j+1}, \dots, i_{j+n} ; \sigma_{j+n+1}) \ 
|  \ i_k \in I'_{k \pmod 2},\ \sigma_{j+n+1} \in \Gamma'_{j+n+1\pmod2} \ \}  \ \rangle.
\end{multline*}
Note that, $Q_j(n)$ is isomorphic to a direct sum of copies of $\Gamma'_{j+n+1 \pmod2}$. \\ 
Let's also denote 
$$Q_j[n] \ = \ \langle \ Q_j(0) \cup Q_j(1) \cup \dots \cup Q_j(n) \ \rangle.$$
Then, it is easy to see by the construction that $Q_j(1)$ is isomorphic to the direct sum of $\#(I'_j)$ coppies 
of $\Gamma'_{j+1 \pmod2}$ and that $Q_j[1] \cong \Gamma'_{j+1 \pmod2} \wr_{I'_j} \Gamma'_j$. 
Relation (R3) immediately implies that $Q_j(n) \vartriangleleft Q_j[n]$ and that there is an extension 
\begin{equation} \label{ExtQj}
\{ 1 \} \longrightarrow Q_j(n) \longrightarrow Q_j[n] \longrightarrow Q_j[n-1] \longrightarrow \{ 1 \}.
\end{equation}
We can actually write  	
\begin{multline*} 
Q_j[n] \ \cong  \\ 
\Gamma'_{j+n \pmod2} \wr_{I'_{j+n-1 \pmod2}} \Gamma'_{j+n-1 \pmod2} \wr_{I'_{j+n-2 \pmod2}} \dots  
\wr_{I'_{j+1 \pmod2}} \Gamma'_{j+1 \pmod2} \wr_{I'_j} \Gamma'_j.
\end{multline*}
Since we have the natural embeddings $Q_j[m] \hookrightarrow Q_j[n]$ for $0 \leq m \leq n$, then it is clear that 
$Q_j$ is the direct limit group of the groups $Q_j[n]$, i.e., 
\begin{equation} \label{Qjlim}
Q_j \ = \ \underset{\underset{n}{\longrightarrow}}{\lim} \ Q_j[n].
\end{equation}
Now, we observe that
\begin{lem} \label{Qjamenable}
$Q_0$ is amenable if and only if $Q_1$ is amenable, if and only if both groups $\Gamma'_0$ and $\Gamma'_1$ are amenable. 
\end{lem}

\begin{proof}
If we suppose that, say, $\Gamma'_0$ is not amenable, then, by equation (\ref{recursive}) for $j=0$, it will follow that $Q_0$ is not amenable, and, 
by equation (\ref{recursive}) for $j = 1$, it will follow that $Q_1$ is not amenable. 
\par
Conversely, suppose that both $\Gamma'_0$ and $\Gamma'_1$ are amenable. Let's note that $Q_j(n)$ is amenable being isomorphic to a direct sum of 
copies of the group $\Gamma'_{j+n+1 \pmod2}$. Also, $Q_j[0] = Q_j(0) \cong \Gamma'_j$ is amenable, and, by equation (\ref{ExtQj}) for 
$n=1$, it follows that $Q_j[1]$ is also amenable. 
\par
Equation (\ref{ExtQj}) and an easy induction establish the amenability of $Q_j[n]$ for each $n \in \mathbb{Z}_0$. 
\par
Finally, we see from the direct limit representation of $Q_j$ (equation (\ref{Qjlim})) that $Q_j$ is amenable. 
\end{proof} 

\subsection{Group-Theoretic Structure} 

First, we need an easy Lemma re-expressing our conditions. 

\begin{lem} \label{expression}
Let $G$ be a group, and let $H$ be its subgroup. Assume that $G = \langle g H g^{-1} \ | \ g \in G \rangle$. Then: \\ 
\text{(i)} $G \ = \ \langle H \ \cup \ [G, G] \rangle$. \\ 
\text{(ii)} $[G, G] \ = \ \langle g h g^{-1} h^{-1} \ | \ g \in G,\ h \in \underset{a \in G}{\cup} a H a^{-1} \rangle$.
\end{lem}

\begin{proof}
(i) Every element $g = a_1 h_1 a_1^{-1} a_2 h_2 a_2^{-1} a_3 h_3 a_3^{-1} \cdots a_n h_n a_n^{-1}$ can be written as 
$$g = (a_1 h_1 a_1^{-1} h_1^{-1}) h_1 (a_2 h_2 a_2^{-1} h_2^{-1}) h_2 \cdots (a_n h_n a_n^{-1} h_n^{-1}) h_n.$$ 
(ii) Take $g \in G$ and $g' = h_1 \dots h_n$, where $h_i \in \underset{a \in G}{\cup} a H a^{-1}$. Then $g' = h_1 \cdot h'$, where 
$h' = h_2 \cdots h_n$ is a product of $n-1$ members of $\underset{a \in G}{\cup} a H a^{-1}$. Then, 
\begin{multline*} 
g g' g^{-1} (g')^{-1} = g h_1 h' g^{-1} (h')^{-1} h_1^{-1} = (g h_1 g^{-1} h_1^{-1}) h_1 (g h' g^{-1} (h')^{-1}) h_1^{-1} = \\ 
(g h_1 g^{-1} h_1^{-1}) \cdot (h_1 g h_1^{-1}) (h_1 h' h_1^{-1}) (h_1 g h_1^{-1})^{-1} (h_1 h' h_1^{-1})^{-1}
\end{multline*}
In the last expression, $h_1 h' h_1^{-1}$ is a product of $n-1$ members of $\underset{a \in G}{\cup} a H a^{-1}$. A simple induction 
completes the proof.
\end{proof} 

\begin{thm} \label{amalgamstructure} 
Take a group $G = G[I_0,I_1;\iota_0,\iota_1;\Gamma_0,\Gamma_1]$. Assume that: \\ 
\text{(i)} $\Gamma_0$ and $\Gamma_1$ are $2$-transitive, that is, all stabilizers 
$(\Gamma_k)_{i_k}$ are transitive on the sets $I_k \setminus \{ i_k \}$ for all $i_k \in I_k,\ k=0,1$. \\ 
\text{(ii)} Either $\Gamma_k = \langle (\Gamma_k)_{i_k} \ | \ i_k \in I_k \rangle$, or $\Gamma_k = \text{Sym}(2)$ for $k = 0,1$.
\par
Then, there is a group extension 
$$1 \longrightarrow N \longrightarrow G \overset{\theta}{\longrightarrow} 
(\Gamma_0 / [\Gamma_0, \Gamma_0]) \times (\Gamma_1 / [\Gamma_1, \Gamma_1])
\longrightarrow 1,$$
where $N$ is a simple, normal subgroup of $G$ and where $\theta$ is defined on the generators by 
$$\theta(g_0(\sigma_0)) = ([\sigma_0]_0,1),\ \theta(g_1(\sigma_1)) = (1, [\sigma_1]_1), \text{ and}$$ 
$$\theta(h_j( i_j, i_{j+1}, \dots, i_{j+n} ; \sigma_{j+n+1})) = 
\left\{\begin{array}{cl}([\sigma_{j+n+1}]_0,1)&\text{if } j+n+1 \text{ is even} \\ 
(1, [\sigma_{j+n+1}]_1)&\text{if } j+n+1 \text{ is odd.}
\end{array}\right.$$
Here, $[\sigma]_k$ denotes the image of the permutation $\sigma \in \Gamma_k$ in $\Gamma_k / [\Gamma_k, \Gamma_k]$ for $k = 0,1$.
\end{thm} 

\begin{proof} 
First, we can define homomorphisms $\theta_0 : G_0 \to \Gamma_0 / [\Gamma_0, \Gamma_0]$ and 
$\theta_1 : G_1 \to \Gamma_1 / [\Gamma_1, \Gamma_1]$ given on the generators by 
$$\theta_0(g_0(\sigma_0)) = ([\sigma_0]_0,1),\ \theta_1(g_1(\sigma_1)) = (1, [\sigma_1]_1), \text{ and}$$ 
$$\theta_k(h_j( i_j, i_{j+1}, \dots, i_{j+n} ; \sigma_{j+n+1})) = 
\left\{\begin{array}{cl}([\sigma_{j+n+1}]_0,1)&\text{if } j+n+1 \text{ is even} \\ 
(1, [\sigma_{j+n+1}]_1)&\text{if } j+n+1 \text{ is odd,}
\end{array}\right.$$
where $k = 0,1$. This is possible since the respective commutators are in the kernels. Next, observe that $\theta_0 |_H = \theta_1 |_H$, and due to the 
universal property of the group amalgamations (Remark \ref{amalgamuniversal}), $\theta$ can be defined.
\par
Observe that, according to the definition of $\theta$, for $g \in G$, we have \\ 
$$ \theta(g) = \left( \prod \{ [\sigma]_0 \ | \ \sigma \in \Gamma_0 \text{ is presented in } g \},\ 
\prod \{ [\tau]_1 \ | \ \tau \in \Gamma_1 \text{ is presented in } g \} \right). $$
Therefore, by Lemma \ref{generation} easily follows that $N$ is generated by the set 
\begin{multline*} 
\{ g_j(\sigma_j) \ | \ \sigma_j \in \Gamma_j, \ [\sigma_j]_j = 1, \ j = 0,1 \} \ \cup \\ 
\{ g_j(\sigma_j) h_k(i_k, \dots, i_{2l+j-1}; \tau_{j+2l}) \ | 
\ l \in \mathbb{N}_0,\ \sigma_j, \tau_{j+2l} \in \Gamma'_j,\  [\sigma_j \tau_{j+2l}]_j = 1,\ i_m \in I'_{m \pmod 2}, \ j,k = 0,1 \} \ 
\cup \\ \{ h_j(i_j, \dots, i_{2t+j-1}; \sigma_{j+2t}) h_k(s_k, \dots, s_{2l+j-1}; \tau_{j+2l}) \ | \\ 
\ [\sigma_{j+2t} \tau_{j+2l}]_j = 1,\ t,l \in \mathbb{N}_0,\ j,k = 0,1,\ \sigma_{j+2t}, \tau_{j+2l}  \in \Gamma'_j, 
i_m, s_m \in I'_{m \pmod2} \}. 
\end{multline*}
Consequently, it is easy to see that the following set generates $N$
\begin{multline} \label{Ngeneration} 
\{ g_k(\sigma_k) \ | \ \sigma_k \in \Gamma_k, \ [\sigma_k]_k = 1, \ k=0,1 \} \ \cup \\ 
\{ g_k(\sigma_k) h_{k+1 \pmod 2}(i_{k+1 \pmod 2}; \sigma_k^{-1}) \ | \ \sigma_k \in \Gamma'_k, \ k=0,1 \} \ \cup \\ 
\{ h_0(i_0, \dots, i_n; \sigma_{n+1}) h_1(s_1, \dots, s_n; \sigma_{n+1}^{-1}) \ | \ n \in \mathbb{N}, \ \ i_k, s_k \in I'_{k \pmod2}, 
\sigma_{n+1} \in \Gamma'_{n+1 \pmod 2} \} \ \cup \\  
\{ h_0(i_0, \dots, i_{n-1} ; \sigma_n) h_1(s_1, \dots, s_{n+1} ; \sigma_n^{-1}) \ | \ n \in \mathbb{N}, \ i_k, s_k \in I'_{k \pmod2}, \ 
\sigma_n \in \Gamma'_{n \pmod 2} \}. 
\end{multline}

Now, let $a \in N \setminus \{ 1 \}$ be arbitrary. We need to show that the normal closure, 
$\langle \langle \ a \ \rangle \rangle_N$, of $a$ in $N$ coincides with $N$. 
Relation (R3) shows that shorter $h$'s modify longer ones, so this observation, together with relation (R6) show that, for an element 
$h_0(i_0, \dots, i_m; \sigma_{m+1})$ with a large enough length $m \in \mathbb{N}$ and appropriate $i_k$'s, we have 
$$a h_0(i_0, \dots, i_m; \sigma_{m+1}) = h_j(s_j, \dots, s_{2n+m}; \sigma_{m+1}) a$$ 
for some $j \in \{ 0,1 \}$, some $n \in \mathbb{Z}$, and some $s_k$'s. Therefore, after observing that, for appropriate $i_k$'s, 
$h_0(i_0, \dots, i_m; \sigma_{m+1})$ and $h_j(s_j, \dots, s_{2n+m}; \sigma_{m+1})$ commute (and are different), it is easy to see that, 
for this choice of $i_k$'s, we obtain the following element of $\langle \langle \ a \ \rangle \rangle_N$ (let's call it $a'$) 
\begin{multline*} 
h_0(i_0, \dots, i_m; \sigma_{m+1}) h_j(s_j, \dots, s_{2n+m}; \sigma_{m+1}^{-1}) a h_j(s_j, \dots, s_{2n+m}; \sigma_{m+1}) 
 h_0(i_0, \dots, i_m; \sigma_{m+1}^{-1}) a^{-1} = \\ 
h_0(i_0, \dots, i_m; \sigma_{m+1}) h_j(s_j, \dots, s_{2n+m}; \sigma_{m+1}^{-1}) 
h_l(t_l, \dots, t_{2p+m}; \sigma_{m+1}^{-1}) h_j(s_j, \dots, s_{2n+m}; \sigma_{m+1}^{-1}),
\end{multline*}
for some $l \in \{ 0,1 \}$, some $p \in \mathbb{Z}$, and some $t_k$'s. Next, take $n > m$, an element $h_j(x_j, \dots, x_{j+n}; \sigma)$ that doesn't 
commute with $a'$, and a $h_k(y_k, \dots, y_{j+n+2w}; \sigma^{-1})$ that does commute with $a'$. Then, we obtain the following element of 
$\langle \langle \ a \ \rangle \rangle_N$
\begin{multline*} 
h_j(x_j, \dots, x_{j+n}; \sigma) h_k(y_k, \dots, y_{j+n+2w}; \sigma^{-1}) a' 
h_k(y_k, \dots, y_{j+n+2w}; \sigma) h_j(x_j, \dots, x_{j+n}; \sigma^{-1}) (a')^{-1} \\ 
= h_j(x_j, \dots, x_{j+n}; \sigma) h_j(x'_j, \dots, x'_{j+n}; \sigma^{-1}) \ \equiv \ a''. 
\end{multline*}
Then, by relation (R6) it is easy to see that we can find $\gamma \in G$, and after eventually multiplying $\gamma$ by an element of the form 
$g_0(\sigma'_0) g_1(\sigma'_1)$, we can have $\gamma \in N$, and, finally, obtain the following element of 
$\langle \langle \ a \ \rangle \rangle_N$:
$$\gamma a'' \gamma^{-1} = 
g_{j + n + 1 \pmod2}(\sigma) h_r(z_r, \dots, z_{2q + j + n}; \sigma^{-1})$$ 
for appropriate $r \in \{ 0,1 \},\ \eps \in \{ -1, 1 \}$, $z_k$'s, and arbitrary $\sigma \in \Gamma'_{j + n + 1 \pmod2}$. 
Since $n \in \mathbb{N}$ is also a large, arbitrary number, there are elements 
$$g_0(\sigma'_0) h_r(z_r, \dots, z_{2q-1}; (\sigma'_0)^{-1}),\ g_1(\sigma'_1) h_d(z'_d, \dots, z'_{2u}; (\sigma'_1)^{-1})\ \in \  
\langle \langle \ a \ \rangle \rangle_N,$$ 
where $\sigma'_0 \in \Gamma'_0$ and $\sigma'_1 \in \Gamma'_1$ are arbitrary. 
\par 
Take $\sigma_0 \in \Gamma_0 \setminus \Gamma_0'$ and consider 
$f = g_0(\sigma_0) h_r(z'_r, z_{r+1}, \dots, z_{2q-1}; (\sigma'_0)^{-1})$, where $z''_r = z_r$ if $r=1$ and 
$z''_r = \sigma_0^{-1} \circ \sigma_0'(z_r)$ if $r=0$. Then, 
$$b \equiv f^{-1} g_0(\sigma'_0) h_r(z_r, \dots, z_{2q-1}; (\sigma'_0)^{-1}) f = 
g_0(\sigma_0^{-1} \sigma'_0 \sigma_0) h_r(z''_r, z_{r+1}, \dots, z_{2q-1}; (\sigma'_0)^{-1}) \in \langle \langle \ a \ \rangle \rangle_N.$$
Let $n \geq 1$, and let $w_0 = \sigma_0^{-1} (\sigma'_0)^{-1} \sigma_0 (\iota_0) \in I_0'$. Choose $j_k$'s and $v_k$'s such that 
$h_0(w_0, j_1, \dots, j_n ; \tau_{n+1})$ and $h_0(\sigma_0^{-1}(\iota_0), v_1, \dots, v_{2m+n}; \tau_{n+1}^{-1})$ commute with 
$h_r(z'_r, z_{r+1}, \dots, z_{2q-1}; (\sigma'_0)^{-1})$ 
(by altering $\sigma_0$, $j$'s, $v$'s, or even the $\gamma$ above if needed). Observe that, 
$h_0(\sigma_0^{-1}(\iota_0), v_1, \dots, v_{2m+n}; \tau_{n+1}^{-1})$ commutes with $b$. Define 
$$c = h_0(w_0, j_1, \dots, j_n ; \tau_{n+1}) h_0(\sigma_0^{-1}(\iota_0), v_1, \dots, v_{2m+n}; \tau_{n+1}^{-1}) \ \in \ N.$$
Then, $d = c b c^{-1} b^{-1}= h_0(w_0, j_1, \dots, j_n ; \tau_{n+1}) h_1(j_1, \dots, j_n; \tau_{n+1}^{-1}) \ \in \ \langle \langle \ a \ \rangle \rangle_N.$
\par
In the above expression $\tau_{n+1} \in \Gamma_{n+1 \pmod 2}'$ is arbitrary. Relation (R3) can be used repeatedly on $d$ to infer that 
the third set of (\ref{Ngeneration}) belongs to $\langle \langle \ a \ \rangle \rangle_N$. 
\par 
The element $g_1(\sigma'_1) h_d(y_d, \dots, y_{2u}; (\sigma'_1)^{-1}) \in \langle \langle \ a \ \rangle \rangle_N$ can be used in analogous way to 
infer that the fourth set of (\ref{Ngeneration}) belongs to $\langle \langle \ a \ \rangle \rangle_N$. 
\par 
The second set of (\ref{Ngeneration}) is formed by products of the third and the fourth set of (\ref{Ngeneration}) with the elements 
$$g_0(\sigma'_0) h_r(z_r, \dots, z_{2q-1}; (\sigma'_0)^{-1}),\text{  and  } g_1(\sigma'_1) h_d(z'_d, \dots, z'_{2u}; (\sigma'_1)^{-1}),$$
so it also belongs to (\ref{Ngeneration}). 
\par 
Finally, we need to show that the first set of (\ref{Ngeneration}) is a subset of $\langle \langle \ a \ \rangle \rangle_N$. Let $k \in \{ 0,1 \}$. If 
$\Gamma_k = \text{Sym}(2)$, then $[\Gamma_k, \Gamma_k]$ is trivial, so there is nothing to prove. Let's assume that 
$\Gamma_k \not= \text{Sym}(2)$. In this situation, $I'_k$ has at least two elements. 
\par
Now, take arbitrary $\gamma_k, \tau_k \in \Gamma_k$ with representations (from Lemma \ref{expression} (i)) 
$$\gamma_k^{-1} \tau_k \gamma_k \ = \ a_k^1 b_k^1 \cdots a_k^n b_k^n, \text{ where } a_k^j \in [\Gamma_k, \Gamma_k],\ b_k^j \in \Gamma_k',$$ 
$$\gamma_k \ = \ c_k^1 d_k^1 \cdots c_k^m d_k^m, \text{ where } c_k^j \in [\Gamma_k, \Gamma_k],\ d_k^j \in \Gamma_k',$$
and an arbitrary $\sigma_k \in \Gamma_k'$. Consider 
$$p_k = h_k(i_k, i_{k+1 \pmod 2}, i_k, i_{k+1 \pmod 2}; b_k^n) g_k(a_k^n b_k^n) \ \in \ N \text{  and}$$ 
$$q_k = g_k(\sigma_k) h_k(i'_k, i_{k+1 \pmod 2}, i''_k, i_{k+1 \pmod 2}; (\sigma_k)^{-1}) \ \in \ \langle \langle \ a \ \rangle \rangle_N,$$
where $i_{k+1 \pmod 2} \in I'_{k+1 \pmod2}, \ i_k, i'_k, i''_k \in I'_k, \ 
i'''_k \equiv a_k^n b_k^n \sigma_k^{-1} (a_k^n b_k^n)^{-1}(i_k) \not= \iota_k,$ and where $a_k^n b_k^n(i'_k) \not= \iota_k$. Then, 
\begin{multline*} 
\langle \langle \ a \ \rangle \rangle_N \ \ni \ p_k q_k p_k^{-1} \ = \\ 
g_k(a_k^n b_k^n \sigma_k (a_k^n b_k^n)^{-1}) h_k(a_k^n b_k^n(i'_k), i_{k+1 \pmod 2}, i''_k, i_{k+1 \pmod 2}; (\sigma_k)^{-1}) \cdot \\ 
\cdot  h_k(i'''_k, i_{k+1 \pmod 2}, i_k, i_{k+1 \pmod 2}; b_k^n)  
h_k(i_k, i_{k+1 \pmod 2}, i_k, i_{k+1 \pmod 2}; (b_k^n)^{-1}). 
\end{multline*}
In the last expression, the product of the last two factors belongs to $\langle \langle \ a \ \rangle \rangle_N$, and therefore 
$$g_k(a_k^n b_k^n \sigma_k (a_k^n b_k^n)^{-1}) h_k(a_k^n b_k^n(i'_k), i_{k+1 \pmod 2}, i''_k, i_{k+1 \pmod 2}; (\sigma_k)^{-1}) \ 
\in \ \langle \langle \ a \ \rangle \rangle_N.$$
Next, we can play the same game with $p_k$ replaced by 
$$h_k(i_k, i_{k+1 \pmod 2}, i_k, i_{k+1 \pmod 2}; b_k^{n-1}) g_k(a_k^{n-1} b_k^{n-1})$$ 
and $q_k$ replaced by 
$$g_k(a_k^n b_k^n \sigma_k (a_k^n b_k^n)^{-1}) h_k(i'_k, i_{k+1 \pmod 2}, i''_k, i_{k+1 \pmod 2}; (\sigma_k)^{-1})$$ 
for appropriate $i$'s. After $n$ steps, we will get 
$$g_k(\gamma_k^{-1} \tau_k \gamma_k \sigma_k \gamma_k^{-1} \tau_k^{-1} \gamma_k) 
h_k(i'_k, i_{k+1 \pmod 2}, i''_k, i_{k+1 \pmod 2}; (\sigma_k)^{-1}) \ \in \  \langle \langle \ a \ \rangle \rangle_N.$$
Finally, after multiplying the last element on the right by 
$$h_k(i'_k, i_{k+1 \pmod 2}, i''_k, i_{k+1 \pmod 2}; \sigma_k) g_k(\sigma_k^{-1}) \ \in \  \langle \langle \ a \ \rangle \rangle_N,$$ 
we will arrive at 
$$g_k(\gamma_k^{-1} \tau_k \gamma_k \sigma_k \gamma_k^{-1} \tau_k^{-1} \gamma_k \sigma_k^{-1}) \ \in \ \langle \langle \ a \ \rangle \rangle_N.$$
This procedure can be applied again using $c_k^m d_k^m,\ \dots \ , c_k^1 d_k^1$, and conclude that 
$$g_k(\tau_k \cdot (\gamma_k \sigma_k \gamma_k^{-1}) \cdot \tau_k^{-1} \cdot (\gamma_k \sigma_k^{-1} \gamma_k^{-1} )) 
\ \in \ \langle \langle \ a \ \rangle \rangle_N.$$
Remembering that, in the last expression, $\tau_k$ and $\gamma_k$ were arbitrary elements of $\Gamma_k$ and $\sigma_k$ was an arbitrary 
element of $\Gamma_k'$, and using Lemma \ref{expression} (ii), we conclude that for, $k = 0,1$, 
$$\{ g_k(\sigma_k) \ | \ [\sigma_k]_k = 1,\ \sigma_k \in \Gamma_k \} \ \subset \ \langle \langle \ a \ \rangle \rangle_N.$$
\par 
We have established that $\langle \langle \ a \ \rangle \rangle_N$ contains all sets from (\ref{Ngeneration}), and therefore 
$N = \langle \langle \ a \ \rangle \rangle_N.$
\end{proof}

\begin{remark} 
Note that, the example introduced in \cite[Section 4]{IO} corresponds to the special case $\Gamma_0 \cong \Gamma_1 \cong \text{Sym}(3)$, and 
\cite[Proposition 4.5]{IO} is the respective version of Theorem \ref{amalgamstructure}. 
\end{remark}

\begin{remark} \label{edgecommutants}
This remark is related to the simplicity criteria of \cite[Section 4]{leboudec16} and uses notations thereof. It is easy to see that our groups satisfy 
the edge independence property, so \cite[Corollary 4.6]{leboudec16} can be applied. It follows that $N$ contains 
the group $\langle [G_e, G_e] \ | \ e \text{ is an edge of } T \rangle$. Since the last group is normal in $G$ and since $N$ is simple, it follows that 
$$ N \ = \ \langle [G_e, G_e] \ | \ e \text{ is an edge of } T \rangle.$$ 
Last conclusion can also be obtained directly. 
\end{remark}

\subsection{Analytic Structure} 

\begin{lem} \label{generaltype}
The action of each group $G = G[I_0,I_1;\iota_0,\iota_1;\Gamma_0,\Gamma_1]$ on its Bass-Serre tree is minimal and of general type. 
\end{lem}

\begin{proof}
Since the action is transitive, it is minimal. Also, the Bass-Serre tree is not a linear tree by Remark \ref{nondegenerate}. 
The result now follows from \cite[Proposition 19 (ii)]{Harpepreaux}. 
\end{proof}

\begin{thm} \label{amalgamC*simple} 
The amalgamated free product $G = G[\Gamma_0,\Gamma_1]$ has the unique trace property. It is $C^*$-simple if and only if either one of the groups 
$\Gamma'_0$ or $\Gamma'_1$ is non-amenable.
\end{thm}

\begin{proof}
Since $G$ is a nondegenerate amalgam by Remark \ref{nondegenerate}, Proposition \ref{KjQj} (ii) and \cite[Proposition 3.1]{IO} establish 
the first part. Since the action of $G$ on its Bass-Serre tree is minimal and of general type by Lemma \ref{generaltype}, \cite[Theorem 3.9]{BIO}, 
Proposition \ref{KjQj} (i), Proposition \ref{KjTj}, and Lemma \ref{Qjamenable} establish the second part. 
\end{proof}

Now, we prove 

\begin{thm} 
The amalgamated free product $G = G[\Gamma_0,\Gamma_1]$ in not inner amenable. 
\end{thm}

\begin{proof}
Lemma \ref{generaltype} allows us to apply Proposition \ref{fledged}. Therefore, we need to show that the action of 
$G = G[I_0,I_1;\iota_0,\iota_1;\Gamma_0,\Gamma_1]$ on its Bass-Serre is finitely fledged. 
\par
For this, take any	 elliptic element $g \in G \setminus \{ 1 \}$. Since $g$ fixes some vertex, it is a conjugate of an element of $G_j$, 
where either $j=0$, or $j=1$. The finite fledgedness property is conjugation invariant, so we can assume $g \in G_j \setminus \{ 1 \}$. 
\par
From Lemma \ref{generation} (ii) and (iii), we can write $g = g_j(\sigma_j) h_0 h_1$, where $\sigma_j \in \Gamma_j$, 
$$h_0 = \prod_{k=1}^{m} h_0(i_0^k, \dots, i_{0+n_k}^k; \theta'_k), \ \ h_1 = \prod_{l=m+1}^{r} h_1(i_1^l, \dots, i_{1+n_l}^l; \xi'_l),$$ 
$r \geq m \geq 0, \ \theta'_k \in \Gamma'_{j+n_k+1 \pmod2},\ \xi'_l \in \Gamma'_{j+n_l+1 \pmod2}$, and $i_z^p \in I'_{z \pmod2}$. 
We also require $0 \leq n_1 \leq \dots \leq n_m $ and $0 \leq n_{m+1} \leq \dots \leq n_r$. 
\par
Let's assume that $g$ fixes a vertex $v = v(\iota_t, i_t, \dots, i_{t+n})$, where $n \geq \max \{ n_m, n_r \} + 1$, and take 
$w = v(\iota_t, i_t, \dots, i_{t+n}, j_{t+n+1}, \dots, j_{t+n+d})$ for any $d \geq 1$ and any $j_k \in I'_{k \pmod 2}$. We note that, $h_{t + 1 \pmod 2}$ fixes $w$ and $h_t$ modifies only indices with numbers no greater than $\{ n_m, n_r \} + 1 \leq n$. Therefore, 
$$h_t v = v(\iota_t, i'_t, \dots, i'_{t+n}) \text{     and     } h_t w = v(\iota_t, i'_t, \dots, i'_{t+n}, j_{t+n+1}, \dots, j_{t+n+d}),$$ 
for some $i'_k \in I'_{k \pmod 2}$. By our assumption, it follows that 
$$v \ = \ g_j(\sigma_j) h_t v \ = \ g_j(\sigma_j)v(\iota_t, i'_t, \dots, i'_{t+n}).$$ 
Form the way in which $g_j(\sigma_j)$ acts on the vertices, it follows that $v = v(\iota_t, i'_t, \dots, i'_{t+n})$ and that $g_j(\sigma_j)v = v$. 
Consequently, 
$$h_t w \ = \ w \ \ \ \text{and} \ \ \ g_j(\sigma_j)w \ = \ w,$$
so $gw = w$. 
\par
It is easy to see that this concludes the proof.
\end{proof}

\begin{remark}
Clearly, Theorems \ref{amalgamC*simple} and \ref{amalgamstructure} imply: \\ 
If either $\Gamma_0$ or $\Gamma_1$ is non-amenable, then the amenablish radical of $G$ is trivial. \\ 
If $\Gamma_0$ and $\Gamma_1$ are both amenable, then, since the class of amenablish groups is closed under extensions, it follows that $G$ is 
amenablish. 
\end{remark}

\section{HNN-Extensions} 

\subsection{Notation, Definitions, Quasi-Kernels} 

We use notation, some of which appear in \cite{BIO}: 
\par 
$$T_\eps = \{ \gamma = g_0 \tau^\eps g_1\tau^{\eps_1}\cdots g_n\tau^{\eps_n}g_{n+1} \ | \ n \geq 0, \gamma \in \Lambda 
\text{ is reduced} \},$$ 
$$T_\eps^\dagger = \{ \gamma = \tau^\eps g_1\tau^{\eps_1}\cdots g_n\tau^{\eps_n}g_{n+1} \ | \ n \geq 0, \gamma \in \Lambda 
\text{ is reduced} \}.$$
For $\eps = \pm 1$, consider also the quasi-kernels defined in \cite{BIO}: 
\begin{equation} \label{K_eps} 
K_\eps \equiv \bigcap_{r \in \Lambda \setminus T_\eps^\dagger}  r H r^{-1}. 
\end{equation}
They satisfy the relation $\ker \Lambda = K_1 \cap K_{-1}$, where, by definition, 
$$\ker \Lambda \equiv  \bigcap_{r\in\Lambda}rHr^{-1}.$$ 
\par
It follows from \cite[Theorem 4.19]{BIO} that $\Lambda$ has the unique trace property if and only if $\ker \Lambda$ has the unique trace property. 
It also follows from \cite[Theorem 4.20]{BIO} that $\Lambda$ is $C^*$-simple if and only if $K_{-1}$ or $K_1$ is trivial or non-amenable provided 
$\Lambda$ is a non-ascending HNN-extension and $\ker \Lambda$ is trivial. 
\par
 We need the following results. 

\begin{remark} \label{Bass-Serre-HNN} 
Consider the Bass-Serre tree $\Theta = \Theta[\Lambda]$ of the group 
$$\Lambda = \operatorname{HNN}(G,H,\theta)=\langle G, \tau\mid\tau^{-1} h\tau=\theta(h)\text{ for all }h\in H\rangle,$$ 
and consider the edge $H$ connecting 
vertices $G$ and $\tau G$. Denote by $\Theta_1$ the full subtree of $\Theta$ consisting of all vertices $v \in \Theta$ satisfying 
$\dist(v, G) < \dist(v, \tau G)$. Also, denote by $\bar{\Theta}_1$ the full subtree of $\Theta$ consisting of all vertices $v \in \Theta$ satisfying 
$\dist(v, G) > \dist(v, \tau G)$. Likewise, consider the edge $\tau^{-1} H$ connecting vertices $G$ and $\tau^{-1} G$. Then, denote by $\Theta_{-1}$ 
the full subtree of $\Theta$ consisting of all vertices $v \in \Theta$ satisfying $\dist(v, G) < \dist(v, \tau^{-1} G)$. 
Also, denote by $\bar{\Theta}_{-1}$ the full subtree of $\Theta$ consisting of all vertices $v \in \Theta$ satisfying $\dist(v, G) > \dist(v, \tau^{-1} G)$. 
\par
It is easy to see that $\bar{\Theta}_\eps  = \tau^\eps \Theta_{-\eps}$, 
$$\Theta_\eps = \{ G \} \ \cup \ \{ \ t_\eps G \ | \ t_\eps \in \Lambda \setminus T_\eps^\dagger \ \}, \text{  and  } 
\bar{\Theta}_\eps = \{ \ t_\eps^\dagger G \ | \ t_\eps^\dagger \in T_\eps^\dagger \ \}.$$ 
\end{remark}

\begin{prop} \label{KepsTheta}
With the notation from the previous Remark, the following hold for each $\eps = \pm 1$: \\ 
(i) $K_\eps = \Lambda_{(\Theta_\eps)}$. \\ 
(ii) $K_\eps < H \cap \theta(H)$. \\ 
(iii) $\gamma K_\eps \gamma^{-1} = \Lambda_{( \gamma \Theta_\eps )}$ for every $\gamma \in \Lambda$. 
In particular $\Lambda_{( \bar{\Theta}_\eps )} = \tau^\eps K_{-\eps} \tau^{-\eps}.$ \\ 
\end{prop}

\begin{proof} 
(i) 
\begin{multline*} 
g \in K_\eps \ \ \ \Longleftrightarrow \\ 
r^{-1} g r \in H, \ \ \forall r \in \Lambda \setminus T_\eps^\dagger \ \ \Longleftrightarrow  \ 
g r \in r H, \ \ \forall r \in \Lambda \setminus T_\eps^\dagger \ \ \Longleftrightarrow  \\ 
g r H = r H, \ \ \forall r \in \Lambda \setminus T_\eps^\dagger \ \ \Longleftrightarrow 
\ \ \ g \ \text{fixes every edge of} \ \Theta_\eps \ \ \Longleftrightarrow  \\ 
g \in \Lambda_{(\Theta_\eps)}. 
\end{multline*} 
(ii) From (i), we know that every element $g \in K_\eps$ fixes all vertices adjacent to $G$, except for the vertex $\tau^\eps G$, eventually. 
Therefore, it also fixes $\tau^\eps G$, so $g$ fixes all edges around $G$. In particular, $g$ fixes the edge $H$, so $g \in H$. Likewise, 
$g$ fixes the edge $\tau^{-1} H$, so $g \in \tau^{-1} H \tau = \theta(H).$ \\ 
(iii) As in (i), we have
\begin{multline*} 
g \in \gamma K_\eps \gamma^{-1} \ \ \Longleftrightarrow \\ 
\gamma^{-1} g \gamma \in K_\eps \ \ \Longleftrightarrow \ \ 
\gamma^{-1} g \gamma \in \Lambda_{(\Theta_\eps)} \ \ \Longleftrightarrow \  \ 
g  \in \gamma \Lambda_{(\Theta_\eps)} \gamma^{-1} \ \ \Longleftrightarrow \\ 
g  \in  \Lambda_{(\gamma \Theta_\eps)}. 
\end{multline*}
\end{proof}

\begin{lem} \label{K1K-1} 
For $\eps = \pm 1$, $K_\eps$ is a normal subgroup of $H_{-\eps}$, and a normal subgroup of $H \cap \theta(H)$. Moreover, if 
$\ker \Lambda$ is trivial, then $K_{-1}$  and $K_1$ have a trivial intersection and mutually commute.
\end{lem}

\begin{proof}
From Proposition \ref{KepsTheta} (ii), it follows that $K_1$ and $K_{-1}$ are subgroups of $H \cap \theta(H)$. 
Take $h \in H_{-\eps}$. Then, 
\begin{multline*} 
h \cdot T_\eps^\dagger = \{ h \tau^\eps g_1\tau^{\eps_1}\cdots g_n\tau^{\eps_n}g_{n+1} \ | \ n \geq 0, \ 
\tau^\eps g_1\tau^{\eps_1}\cdots g_n\tau^{\eps_n}g_{n+1} \text{ is reduced} \} =  \\ 
\{ \tau^\eps \theta^\eps(h) g_1\tau^{\eps_1}\cdots g_n\tau^{\eps_n}g_{n+1} \ | \ n \geq 0, \ 
\tau^\eps g_1\tau^{\eps_1}\cdots g_n\tau^{\eps_n}g_{n+1} \text{ is reduced} \} =  T_\eps^\dagger.
\end{multline*}
This gives the first assertion. For the second assertion, take $k_\eps \in K_\eps$ for each $\eps = \pm 1$. Then, from 
$K_\eps \triangleleft H \cap \theta(H)$, it follows that $k_{-1} k_1^{-1} k_{-1}^{-1} \in K_1$ and $k_1 k_{-1} k_1^{-1} \in K_{-1}$. Thus, 
$$K_{-1} \ni (k_1 k_{-1} k_1^{-1}) k_{-1}^{-1} = k_1 (k_{-1} k_1^{-1} k_{-1}^{-1}) \in K_1,$$
and therefore $k_1 k_{-1} k_1^{-1}k_{-1}^{-1} \in K_1 \cap K_{-1} = \ker \Lambda = \{ 1 \}$.
\end{proof}

\begin{lem} \label{K(gamma)}
(i) Let $\gamma = \tau^{\eps_n} g_n \cdots g_2 \tau^{\eps_1} g_1 \tau^\eps \in \Lambda$ be reduced. 
Then $\gamma \cdot T_{-\eps}^\dagger \supset T_{-\eps_n}^\dagger$. In particular $K_{-\eps_n} < \gamma K_{-\eps} \gamma^{-1}$. \\ 
(ii) Let $\gamma \in G \setminus H_\eps$. Then $\gamma T_{-\eps}^\dagger \cap T_{-\eps}^\dagger = \emptyset$. In particular 
$\gamma K_{-\eps} \gamma^{-1} \cap K_{-\eps} = \ker \Lambda$. \\ 
(iii) Let $\gamma \in \Lambda$ be a reduced word starting and ending with $\tau^\eps$. Then, 
$T_{-\eps}^\dagger \cap \gamma T_\eps^\dagger = \emptyset$. In particular $K_{-\eps} \cap \gamma K_\eps \gamma^{-1} = \ker \Lambda$.
\end{lem}

\begin{proof}
(i) Observe that 
\begin{equation*}
\begin{aligned}
\gamma \cdot T_{-\eps} \  \supset & \ \{ \gamma \cdot \tau^{-\eps} g_1^{-1} \tau^{-\eps_1} \cdots g_n^{-1} \tau^{-\eps_n} \cdot 
\tau^{-\eps_n} \cdot g_{n+1} \tau^{\eps_{n+1}} g_{n+2} \tau^{\eps_{n+2}} \cdots g_{n+m} \tau^{\eps_{n+m}} g_{n+m+1} \ | \\ 
& m \geq 0, \ \tau^{-\eps_n} g_{n+1} \tau^{\eps_{n+1}} g_{n+2} \cdots g_{n+m} \tau^{\eps_{n+m}} g_{n+m+1} \text{ is reduced} \} = \\  
                     & \{ \lambda = \tau^{-\eps_n} g_{n+1} \tau^{\eps_{n+1}} g_{n+2} \cdots g_{n+m} \tau^{\eps_{n+m}} g_{n+m+1} \ | 
\ m \geq 0, \ \lambda \text{ is reduced} \} = \\ 
                     & \ T_{-\eps_n}.
\end{aligned}
\end{equation*}
The second statement follows from the observation 
$$\gamma \cdot (\Lambda \setminus T_{-\eps}^\dagger) = \Lambda \setminus \gamma T_{-\eps}^\dagger \subset 
\Lambda \setminus T_{-\eps_n}^\dagger.$$
(ii) and (iii) follow easily. 
\end{proof} 

\begin{lem} \label{K()K(')}
Let $\gamma = g_{n+1} \tau^{\eps_n} g_n \cdots g_2 \tau^{\eps_1} g_1 \tau^\eps$, 
$\gamma' = g'_{n+1} \tau^{\eps'_n} g'_n \cdots g'_2 \tau^{\eps'_1} g'_1 \tau^\eps$, and \\ 
$\gamma'' = g''_{n+1} \tau^{\eps''_n} g'_n \cdots g''_2 \tau^{\eps''_1} g''_1 \tau^{-\eps}$ 
be reduced, where $n \geq 0$ and $\eps = \pm 1$. 
Then: \\ 
(i) If $(\gamma')^{-1} \gamma \in H_{-\eps}$, then 
$\gamma K_\eps \gamma^{-1} = \gamma' K_\eps (\gamma')^{-1}$. \\ 
(ii) If $\ker \Lambda$ is trivial and if $(\gamma')^{-1} \gamma \notin H_{-\eps}$, then 
$\gamma K_\eps \gamma^{-1}$ and $\gamma' K_\eps (\gamma')^{-1}$ 
have a trivial intersection and mutually commute. \\ 
(iii) If $\ker \Lambda$ is trivial, then $\gamma K_\eps \gamma^{-1}$ and $\gamma'' K_{-\eps} (\gamma'')^{-1}$ have a trivial intersection 
and mutually commute. 
\end{lem} 

\begin{proof} 
(i) $(\gamma')^{-1} \gamma K_\eps \gamma^{-1} \gamma' = K_\eps$ by Lemma \ref{K1K-1}. \\ 
(ii) If $(\gamma')^{-1} \gamma$ is an element of $G \setminus H_{-\eps}$, then the assertion follows from Lemma \ref{K(gamma)} (ii). 
If $(\gamma')^{-1} \gamma$ starts with $\tau^{-\eps}$ and ends with $\tau^\eps$, then, by Lemma \ref{K(gamma)} (i), it follows that 
$$(\gamma')^{-1} \gamma K_\eps \gamma^{-1} \gamma' < K_{-\eps},$$
which combined with $K_\eps \cap K_{-\eps} = \ker \Lambda = \{ 1 \}$, proves the assertion. \\ 
(iii) Observe that, the reduced form of $(\gamma'')^{-1} \gamma$ starts and ends with $\tau^\eps$, therefore the assertion follows from 
Lemma \ref{K(gamma)} (iii).
\end{proof}
\par
Assume that $\ker \Lambda = \{ 1 \}$.
\par
It follows from Lemma \ref{K()K(')} that, for two reduced words 
$$\gamma = s_{n+1} \tau^{\eps_n} s_n \cdots s_2 \tau^{\eps_1} s_1 \tau^\eps \text{ and } 
\gamma' = t_{n+1} \tau^{\eps'_n} t_n \cdots t_2 \tau^{\eps'_1} t_1 \tau^\eps$$ 
with $s_i, t_i \in S_{-1} \cup S_1$ and $\eps, \eps_i, \eps'_i \in \{ -1, 1 \}$, 
$$\gamma K_\eps \gamma^{-1} = \gamma' K_\eps (\gamma')^{-1}$$ 
if and only if $\gamma = \gamma'$, and this happens if and only if 
$\eps_i = \eps'_i$ and $s_i = t_i, \ \forall i$. 
In the case $\gamma \not= \gamma'$, $\gamma K_\eps \gamma^{-1}$ and $\gamma' K_\eps (\gamma')^{-1}$ have a trivial intersection 
and mutually commute. \\
If $\gamma'' = r_{n+1} \tau^{\eps''_n} r_n \cdots r_2 \tau^{\eps''_1} s_1 \tau^{-\eps}$ is another reduced word, where 
$r_i \in S_{-1} \cup S_1$ and $\eps''_i \in \{ -1, 1 \}$, 
then $\gamma K_\eps \gamma^{-1}$ and $\gamma'' K_{-\eps} (\gamma'')^{-1}$ have a trivial intersection and mutually commute. 
\par
From the above considerations, it follow that
\begin{equation} \label{cal(K)} 
\mathcal{K}(0) \equiv 
\underset{s \in S_{-1}}{\bigoplus} s K_1 s^{-1} \ \oplus \ \underset{t \in S_1}{\bigoplus} t K_{-1} t^{-1} 
\end{equation}
and, for $n \geq 0$, 
\begin{equation} \label{cal(K)} 
\mathcal{K}(n+1) \equiv  
\underset{s_{n+1} \tau^{\eps_n} s_n \cdots s_2 \tau^{\eps_1} s_1 \tau^\eps \text{ reduced }}{\underset{s_i \in S_{-1} \cup S_1,\ \eps_i = \pm 1}
{\underset{\eps = \pm 1}{\bigoplus}}} 
s_{n+1} \tau^{\eps_n} s_n \cdots s_2 \tau^{\eps_1} s_1 \tau^\eps K_\eps 
\tau^{-\eps} s_1^{-1} \tau^{-\eps_1} s_2^{-1} \cdots s_n^{-1} \tau^{-\eps_n} s_{n+1}^{-1}
\end{equation}
are normal subgroups of $G$. \\ 
Also, consider the groups 
$$\mathcal{K}(\eps, 0) \equiv \underset{s \in S_{-\eps}}{\bigoplus} s K_1 s^{-1} \ \oplus \ \underset{t \in S'_\eps}{\bigoplus} t K_{-1} t^{-1},$$ 
which are normal in $H_\eps$ for $\eps = \pm 1$.

\begin{remark}
The group $G$ acts transitively on the vertices $s \tau G$, where $s \in S_{-1}$. It also acts transitively on the vertices $s \tau^{-1} G$, where $s \in S_1$. 
This fact is an important ingredient in the examples below.
\end{remark}

\begin{remark}
It follows from Lemma \ref{K(gamma)} that $K_{-1}$ is isomorphic to a subgroup of $K_1$ and vice-versa. Consequently, $K_{-1} = \{ 1 \}$ 
if and only if $K_1 = \{ 1 \}$. In this situation, $\mathcal{K}(n) = \{ 1 \} \ \forall n \geq 0$. 
\end{remark}

\subsection{A Family of Examples} 

For $\eps = \pm 1$, consider nonempty sets $I'_\eps$, and let $I_\eps \equiv I'_\eps \sqcup \{ \iota_\eps \}$. Also, let $\Sigma_\eps$ be transitive permutation groups on $I_\eps$, and let $\Gamma = \Sigma_{-1} \cdot \Sigma_1$ be the corresponding permutation group on $I_{-1} \sqcup I_1$. 
Let $\Sigma'_\eps \equiv (\Sigma_\eps)_{ \iota_\eps }$ be the respective stabilizer groups, and define 
$\Gamma_\eps \equiv \Gamma_{ \iota_\eps } = \Sigma'_\eps \cdot \Sigma_{-\eps}$. Define 
$$\Lambda[\Sigma_{-1}, \Sigma_1] = \Lambda[I_{-1}, I_1, \iota_{-1}, \iota_1; \Sigma_{-1}, \Sigma_1] = 
\operatorname{HNN}(G,H,\theta)=\langle G, \tau\mid\tau^{-1} h\tau=\theta(h)\text{ for all }h\in H\rangle,$$
where
\begin{equation*}
\begin{aligned}
\underline{H} \equiv 
\langle \{ h(i_1, \eps_1 \dots, i_n, \eps_n; \sigma_n) \ | & \ n \in \mathbb{N},\ \eps_t \in \{-1, 1 \},\ i_t \in I_{ - \eps_t}, \text{ and } 
\sigma_n \in \Gamma_{\eps_n}; \\ 
                                                                              & \text{with the condition} \ i_t \in I'_{ - \eps_t}, \text{ whenever } \eps_t \eps_{t-1} = -1;\ \} \rangle,
\end{aligned}
\end{equation*}
$$H_\eps = \langle \underline{H} \cup \{ h(\sigma_\eps) \ | \ \sigma_\eps \in \Gamma_\eps \} \rangle,\ \eps= \pm 1.$$
Finally, define 
$$G \ = \ \langle H_{-1}, H_1 \rangle \ = \ \langle \underline{H} \cup 
\{ \ h(\sigma) \ | \ \sigma \in \Gamma \} \rangle.$$
\par 
The following relations hold (there are redundancies): \\ 
(R1) Elements $h(\sigma_{-1})$'s and $h(\sigma_1)$'s commute for all $\sigma_\eps \in \Sigma_\eps$, where $\eps = \pm 1$. \\ 
(R2) Let $1 \leq m < n$, $\sigma_n \in \Gamma_{\eps_n}$, and $\sigma'_m \in \Gamma_{e_m}$. If 
$(i_1, \eps_1 \dots, i_m, \eps_m) \not= (j_1, e_1 \dots, j_m, e_m)$, the elements 
$$h(j_1, e_1 \dots, j_m, e_m; \sigma'_m) \text{  and  } h(i_1, \eps_1 \dots, i_m, \eps_m, \dots, i_n, \eps_n; \sigma_n)$$
commute. \\ 
(R3) For $1 \leq m < n$ and $\sigma_t \in \Gamma_{\eps_t}$, the following holds 
\begin{multline*}
h(i_1, \eps_1 \dots, i_m, \eps_m; \sigma_m)  h(i_1, \eps_1 \dots, i_m, \eps_m, i_{m+1}, \eps_{m+1}, \dots, i_n, \eps_n; \sigma_n) 
h(i_1, \eps_1 \dots, i_m, \eps_m; \sigma_m)^{-1} \ = \\  
h( i_1, \eps_1 \dots, i_m, \eps_m, \sigma_m(i_{m+1}), \eps_{m+1}, \dots, i_n, \eps_n; \sigma_n). 
\end{multline*}
(R4) For $\sigma_m, \sigma'_m \in \Gamma_{\eps_m}$, the following holds
$$h(i_1, \eps_1 \dots, i_m, \eps_m; \sigma_m)  h(i_1, \eps_1 \dots, i_m, \eps_m; \sigma'_m) \ = 
h(i_1, \eps_1 \dots, i_m, \eps_m; \sigma_m \sigma'_m).$$ 
(R5) For $\sigma, \sigma' \in \Gamma$, the following holds 
$$h(\sigma) h(\sigma') = h(\sigma \sigma').$$ 
(R6) For $n \in \mathbb{Z}$, $\sigma \in \Gamma$, and $\sigma_n \in \Gamma_{\eps_n}$, the following holds 
$$ h(\sigma)  h(i_1, \eps_1 \dots, i_n, \eps_n; \sigma_n) h(\sigma)^{-1}  \ = \ 	
h( \sigma(i_1), \eps_1, i_2, \eps_2, \dots, i_n, \eps_n; \sigma_n).$$ 
(R7) For $\eps = \pm 1$ and $\sigma_\eps \in \Gamma_\eps$, the following holds 
$$\theta^{-\eps}(h(\sigma_\eps)) \ = ( \tau^{\eps} h(\sigma_\eps) \tau^{-\eps}) \ = h(\iota_{ - \eps}, \eps; \sigma_\eps).$$
(R8) For $\eps = \pm 1,\ n \in \mathbb{N}$, and $\sigma_n \in \Gamma_{\eps_n}$, the following holds 
\begin{multline*} 
\theta^{-\eps}(h(i_1, \eps, i_2, \eps_2, \dots, i_n, \eps_n; \sigma_n)) \ 
= ( \tau^\eps h(i_1, \eps, i_2, \eps_2, \dots, i_n, \eps_n; \sigma_n) \tau^{-\eps}) \ = \\ 
h(\iota_{ - \eps}, \eps, i_1, \eps, i_2, \eps_2 \dots, i_n, \eps_n; \sigma_n). 
\end{multline*}
(R9) For $\eps = \pm 1$, $n \in \mathbb{N}$, and $\sigma_n \in \Gamma_{\eps_n}$ the following holds 
\begin{multline*} 
\theta^{\eps}(h(i_1, \eps \dots, i_n, \eps_n; \sigma_n)) \ = ( \tau^{-\eps} h(i_1, \eps \dots, i_n, \eps_n; \sigma_n) \tau^{\eps}) \ = \\ 
\begin{cases}
h(i_2, \eps_2 \dots, i_n, \eps_n; \sigma_n), \text{ if } i_1 = \iota_{ - \eps}, \\
h(\iota_\eps, - \eps, i_1, \eps \dots, i_n, \eps_n; \sigma_n), \text{ if } i_1 \not= \iota_{ - \eps}.
\end{cases}
\end{multline*}

\subsection{Some Basic Properties of the Examples and Their Quasi-Kernels} 

In this subsection, we fix a group $\Lambda = \Lambda[I_{-1}, I_1, \iota_{-1}, \iota_1; \Sigma_{-1}, \Sigma_1]$.
\par 
First, let's note that $\Index[G : H_\eps] = \#(I_\eps)$ for $\eps = \pm 1$. To see this, recall that $\Sigma_\eps$ acts transitively on $I_\eps$, and, 
for $i \in I_\eps$, choose $\mu_\eps^i \in \Sigma_\eps$ to satisfy $\mu_\eps^i(\iota_\eps) = i$. 
Let's denote $\lambda_\eps^i = h(\mu_\eps^i)$. If $\sigma \in \Sigma_\eps \setminus \Sigma'_\eps$ satisfies 
$\sigma(\iota_\eps) = i$, then $(\mu_\eps^i)^{-1} \circ \sigma(\iota_\eps) = \iota_\eps$. Therefore, 
$(\mu_\eps^i)^{-1} \circ \sigma \in \Sigma'_\eps$, so $h((\mu_\eps^i)^{-1} \circ \sigma) \in H_\eps$. 
It follows that $h(\sigma) \in h(\mu_\eps^i) H_\eps = \lambda_\eps^i H$. Consequently, for each $\eps = \pm 1$, 
\begin{equation} \label{HNNclasses} 
G \ = \ H_\eps \sqcup \underset{i \in I'_\eps}{\bigsqcup} \lambda_\eps^i H_\eps.
\end{equation}

It is easy to see that, in these notations, for $\eps = \pm 1$, the set 
$$S_\eps \ = \ \{ \ \lambda_\eps^i \ | \ i \in I'_\eps \ \} \ \cup \ \{ \ 1 \ \}$$ 
is a left coset representative of $H_\eps$ in $G$. 
\par 
Next, consider the action of $\Lambda$ on its Bass-Serre tree $\Theta = \Theta [\Lambda]$. 
The set of all adjacent vertices to the vertex $G$ is 
$$\{ \ \tau G \ \} \ \cup \ \{\ \lambda_{-1}^i \tau G\ | \ i \in I'_{-1} \ \} \ \cup \ \{\ \tau^{-1} G \ \} \ \cup \ \{ \ \lambda_1^i \ | \ i \in I'_1 \ \}.$$
This set can be indexed by the set $I_{-1} \cup I_1$ in the obvious way: Denote by $v(\emptyset)$ the vertex $G$, by $v(\iota_{-1}, 1)$ the vertex 
$\tau G$, by $v(\iota_1, -1)$ the vertex $\tau^{-1} G$, by $v(i_{-1}, 1)$ the vertex $\lambda_{-1}^{i_{-1}} \tau G$, where $i_{-1} \in I'_{-1}$, and by 
$v(i_1, -1)$ the vertex $\lambda_1^{i_1} \tau^{-1} G$, where $i_1 \in I'_1$.  Denote a general vertex 
$$\lambda_{-\eps_1}^{i_1} \tau^{\eps_1} \cdots \lambda_{-\eps_n}^{i_n} \tau^{\eps_n} G$$ 
by $v(i_1, \eps_1, \dots, i_n, \eps_n)$ for an element 
$\lambda_{-\eps_1}^{i_1} \tau^{\eps_1} \cdots \lambda_{-\eps_n}^{i_n} \tau^{\eps_n} \in \Lambda$ in its normal form, (i.e.) 
$i_t \in I_{- \eps_t}$, and if $\eps_{t-1} \cdot \eps_t = -1$, then $i_t \in I'_{- \eps_t}$. 
\par
With the notation of Remark \ref{Bass-Serre-HNN}, for $\eps = \pm 1$, $\Theta_\eps$ is the full subtree of $\Theta$ containing the vertex 
$v(\emptyset) = G$ and vertices $v(i_1, \eps_1, \dots, i_n, \eps_n),$ where $n \geq 1$ and $(i_1, \eps_1) \not= (\iota_{-\eps}, \eps)$, 
and $\bar{\Theta}_\eps$ is the full subtree of $\Theta$ containing the vertices $v(\iota_{-\eps}, \eps, i_1, \eps_1, \dots, i_n, \eps_n)$, 
where $n \geq 0$. 

\begin{remark} 
It follows from \cite[Exercise VI.3]{baumslag} that our examples are never finitely presented since $H$ is never finitely generated. 
\end{remark}

We continue with

\begin{lem} \label{HNNgeneration}
(i) Let $m \geq 1$, $\sigma_m \in \Gamma_{\eps_m}$, $i_t \in I_{-\eps_t}$, and $\eps \in \{ -1, 1 \}$ satisfying 
$\eps_t \eps_{t-1}=-1 \Rightarrow i_t \in I'_{-\eps_t}$. Then, 
$$h(i_1, \eps_1 \dots, i_m, \eps_m; \sigma_m) 
= \lambda_{-\eps_1}^{i_1} \tau^{\eps_1} \cdots \lambda_{-\eps_m}^{i_m} \tau^{\eps_m}
                                  h(\sigma_m) \tau^{-\eps_m} (\lambda_{-\eps_m}^{i_m})^{-1} \cdots \tau^{-\eps_1} (\lambda_{-\eps_1}^{i_1})^{-1}.$$

(ii) Every element $h$ of $G$ can be written as 
$$h = h (\sigma) \prod_{k=1}^{m} h(i_1^k, \eps_{k, 1}, \dots, i_{n_k}^k, \eps_{k, n_k}; \sigma_k),$$
where $m \geq 1, \ \sigma_k \in \Gamma_{\eps_{k, n_k}}$, $1 \leq n_1 \leq \dots \leq n_m$, and $\sigma \in \Gamma$ satisfying the condition: 
if $n_k = n_{k+a}$ for 
some $1 \geq k \geq m$ and some $a \geq 1$, then 
$$(i_1^k, \eps_{k, 1}, \dots, i_{n_k}^k, \eps_{k, n_k}) \not= (i_1^{k+a}, \eps_{k+a, 1}, \dots, i_{n_{k+a}}^{k+a}, \eps_{k+a, n_{k+a}}).$$
\par
(iii) Every element $g \in T_\eps$ can be written as 
$$g = \lambda_{-\eps}^i \tau^\eps \lambda_{-\eps_1}^{i_1} \tau^{\eps_1} \cdots \lambda_{-\eps_m}^{i_m} \tau^{\eps_m} h,$$
where $h \in G$ and $m \geq 0$.
\end{lem}

\begin{proof}
(i) Follows by repeated applications of relations (R7), (R8), and (R6). \\ 
(ii) Follows by repeated applications of relations (R3) and (R6). \\ 
(iii) Follows by equation (\ref{HNNclasses}) and the structure of HNN-extensions. 
\end{proof} 

\begin{lem} \label{HNNaction} 
Let $n > m \geq 1$ and $\sigma_k \in \Gamma_{\eps_k}$. Then, the following hold 
\begin{multline*}
\ \ \text{(i)} \  h(i_1, \eps_1 \dots, i_m, \eps_m; \sigma_m)  v(i_1, \eps_1, \dots, i_m, \eps_m, i_{m+1}, \eps_{m+1}, \dots, i_n, \eps_n) = \\ 
= v(i_1, \eps_1, \dots, i_m, \eps_m, \sigma_m(i_{m+1}), \eps_{m+1}, \dots, i_n, \eps_n).
\end{multline*}

(ii) $h(i_1, \eps_1 \dots, i_m, \eps_m; \sigma_m) \in \Lambda_{v(i_1, \eps_1 \dots, i_m, \eps_m)}$, and $h(\sigma) \in \Lambda_{v(\emptyset)}$ 
for $\sigma \in \Gamma$. 
\par
(iii) If $\sigma_\eps \in \Gamma_\eps$, then $h(\sigma_\eps) \in \Lambda_{(\bar{\Theta}_{-\eps})} = \tau^{-\eps} K_\eps \tau^\eps$. 
\par
(iv) Let $m \leq n$, and let $h(i_1, \eps_1 \dots, i_n, \eps_n; \sigma_n), h(j_1, e_1 \dots, j_m, e_m; \delta_m) \in \Lambda$. If 
$(i_1, \eps_1 \dots, i_m, \eps_m) \not= (j_1, e_1 \dots, j_m, e_m)$, 
then $h(i_1, \eps_1 \dots, i_n, \eps_n; \sigma_n) \in \Lambda_{v(j_1, e_1 \dots, j_m, e_m)}$ and 
$h(j_1, e_1 \dots, j_m, e_m; \delta_m) \in \Lambda_{v(i_1, \eps_1 \dots, i_n, \eps_n)}$.
\par
(v) $h(i_1, \eps_1 \dots, i_n, \eps_n; \sigma_n) \in \Lambda_{(\bar{\Theta}_\eps)} \ \iff \ (i_1, \eps_1) \not= (\iota_{-\eps}, \eps)$.
\end{lem}

\begin{proof}
(i) First, note that 
$$ \sigma \equiv (\lambda_{-\eps_{m+1}}^{\sigma_m(i_{m+1})})^{-1} \circ h(\sigma_m)  \lambda_{-\eps_{m+1}}^{i_{m+1}} \ \in \ 
\Gamma_{-\eps_{m+1}}$$ 
since it fixes $\iota_{-\eps_{m+1}}$. It follows by Lemma \ref{HNNgeneration} (i) and (iii) that there are $k_t \in I_{\eps_t}$ and a 
$\chi \in H_{\eps_n}$ satisfying $(\tau^{\eps_{m+1}} \cdots \lambda_{-\eps_n}^{i_n} \tau^{\eps_n})^{-1} = 
\chi \tau^{-\eps_n} \lambda_{\eps_{n-1}}^{k_{n-1}} \cdots \lambda_{\eps_{m+1}}^{k_{m+1}} \tau^{-\eps_{m+1}}.$ Therefore, 
\begin{multline*} 
(\tau^{\eps_{m+1}} \cdots \lambda_{-\eps_n}^{i_n} \tau^{\eps_n})^{-1} 
h(\sigma) \tau^{\eps_{m+1}} \cdots \lambda_{-\eps_n}^{i_n} \tau^{\eps_n} = \\ 
\chi \tau^{-\eps_n} \lambda_{\eps_{n-1}}^{k_{n-1}} \cdots \lambda_{\eps_{m+1}}^{k_{m+1}} \tau^{-\eps_{m+1}} h(\sigma) 
\tau^{\eps_{m+1}} (\lambda_{\eps_{m+1}}^{k_{m+1}})^{-1} \cdots (\lambda_{\eps_{n-1}}^{k_{n-1}})^{-1} \tau^{\eps_n} \chi^{-1} = \\ 
\chi h(\iota_{\eps_n}, -\eps_n, k_{n-1}, -\eps_{n-1}, \dots, k_{m+2}, -\eps_{m+2}, \iota_{\eps_m+1}, -\eps_{m+1}; \sigma) \chi^{-1}.
\end{multline*}
Then, Lemma \ref{HNNgeneration} (i) implies
\begin{multline*}
h(i_1, \eps_1 \dots, i_m, \eps_m; \sigma_m)  v(i_1, \eps_1, \dots, i_m, \eps_m, i_{m+1}, \eps_{m+1}, \dots, i_n, \eps_n) = \\ 
\lambda_{-\eps_1}^{i_1} \tau^{\eps_1} \cdots \lambda_{-\eps_m}^{i_m} \tau^{\eps_m}
                                  h(\sigma_m) \tau^{-\eps_m} (\lambda_{-\eps_m}^{i_m})^{-1} \cdots \tau^{-\eps_1} (\lambda_{-\eps_1}^{i_1})^{-1}  
\cdot \lambda_{-\eps_1}^{i_1} \tau^{\eps_1} \cdots \lambda_{-\eps_n}^{i_n} \tau^{\eps_n} G = \\ 
\lambda_{-\eps_1}^{i_1} \tau^{\eps_1} \cdots \lambda_{-\eps_m}^{i_m} \tau^{\eps_m}
                                  h(\sigma_m)  \lambda_{-\eps_{m+1}}^{i_{m+1}} \tau^{\eps_{m+1}} \cdots \lambda_{-\eps_n}^{i_n} \tau^{\eps_n} G = \\ 
\lambda_{-\eps_1}^{i_1} \tau^{\eps_1} \cdots \lambda_{-\eps_m}^{i_m} \tau^{\eps_m} \lambda_{-\eps_{m+1}}^{\sigma_m(i_{m+1})} 
h(\sigma) \tau^{\eps_{m+1}} \cdots \lambda_{-\eps_n}^{i_n} \tau^{\eps_n} G = \\ 
\lambda_{-\eps_1}^{i_1} \tau^{\eps_1} \cdots \lambda_{-\eps_m}^{i_m} \tau^{\eps_m} \lambda_{-\eps_{m+1}}^{\sigma_m(i_{m+1})} 
\tau^{\eps_{m+1}} \cdots \lambda_{-\eps_n}^{i_n} \tau^{\eps_n} \cdot (\tau^{\eps_{m+1}} \cdots \lambda_{-\eps_n}^{i_n} \tau^{\eps_n})^{-1} 
h(\sigma) \tau^{\eps_{m+1}} \cdots \lambda_{-\eps_n}^{i_n} \tau^{\eps_n} G = \\ 
\lambda_{-\eps_1}^{i_1} \tau^{\eps_1} \cdots \lambda_{-\eps_m}^{i_m} \tau^{\eps_m} \lambda_{-\eps_{m+1}}^{\sigma_m(i_{m+1})} 
\tau^{\eps_{m+1}} \cdots \lambda_{-\eps_n}^{i_n} \tau^{\eps_n} \cdot \\ 
\cdot \chi h(\iota_{\eps_n}, -\eps_n, k_{n-1}, -\eps_{n-1}, \dots, k_{m+2}, -\eps_{m+2}, \iota_{\eps_m+1}, -\eps_{m+1}; \sigma) \chi^{-1}G = \\ 
\lambda_{-\eps_1}^{i_1} \tau^{\eps_1} \cdots \lambda_{-\eps_m}^{i_m} \tau^{\eps_m} \lambda_{-\eps_{m+1}}^{\sigma_m(i_{m+1})} 
\tau^{\eps_{m+1}} \cdots \lambda_{-\eps_n}^{i_n} \tau^{\eps_n} G = \\ 
v(i_1, \eps_1, \dots, i_m, \eps_m, \sigma_m(i_{m+1}), \eps_{m+1}, \dots, i_n, \eps_n).
\end{multline*}
(ii) Second claim is obvious. For the first claim, 
\begin{multline*}
h(i_1, \eps_1 \dots, i_m, \eps_m; \sigma_m) v(i_1, \eps_1 \dots, i_m, \eps_m) = \\ 
\lambda_{-\eps_1}^{i_1} \tau^{\eps_1} \cdots \lambda_{-\eps_m}^{i_m} \tau^{\eps_m}
                                  h(\sigma_m) \tau^{-\eps_m} (\lambda_{-\eps_m}^{i_m})^{-1} \cdots \tau^{-\eps_1} (\lambda_{-\eps_1}^{i_1})^{-1} 
\cdot \lambda_{-\eps_1}^{i_1} \tau^{\eps_1} \cdots \lambda_{-\eps_n}^{i_m} \tau^{\eps_m} G = \\ 
\lambda_{-\eps_1}^{i_1} \tau^{\eps_1} \cdots \lambda_{-\eps_m}^{i_m} \tau^{\eps_m} h(\sigma_m) G = \\ 
v(i_1, \eps_1 \dots, i_m, \eps_m).
\end{multline*}
(iii) The fact $\Lambda_{(\bar{\Theta}_{-\eps})} = \tau^{-\eps} K_\eps \tau^\eps$ is stated in Proposition \ref{KepsTheta}. 
Let $n \geq 0$, and let 
$v(\iota_\eps, -\eps, i_1, \eps_1, \dots, i_n, \eps_n) \in \bar{\Theta}_{-\eps}$. By the argument at the beginning of the proof of (i), there are 
$k_t \in I_{\eps_t}$ and a $\chi \in H_{\eps_n}$ satisfying 
\begin{multline*} 
(\tau^{-\eps} \lambda_{-\eps_1}^{i_1} \tau^{\eps_1} \cdots \lambda_{-\eps_n}^{i_n} \tau^{\eps_n})^{-1} 
h(\sigma_\eps) \tau^{-\eps} \lambda_{-\eps_1}^{i_1} \tau^{\eps_1} \cdots \lambda_{-\eps_n}^{i_n} \tau^{\eps_n} = \\ 
\chi h(\iota_{\eps_n}, -\eps_n, k_{n-1}, -\eps_{n-1}, \dots, i_{\eps_1}, -\eps_1, \eps, \iota_{-\eps}; \sigma_\eps) \chi^{-1}.
\end{multline*}
Therefore, 
\begin{multline*}
h(\sigma_\eps)  v(\iota_\eps, -\eps, i_1, \eps_1, \dots, i_n, \eps_n) = \\ 
h(\sigma_\eps) \tau^{-\eps} \lambda_{-\eps_1}^{i_1} \tau^{\eps_1} \cdots \lambda_{-\eps_m}^{i_n} \tau^{\eps_n} G = \\ 
\tau^{-\eps} \lambda_{-\eps_1}^{i_1} \tau^{\eps_1} \cdots \lambda_{-\eps_m}^{i_n} \tau^{\eps_n} \cdot 
(\tau^{-\eps} \lambda_{-\eps_1}^{i_1} \tau^{\eps_1} \cdots \lambda_{-\eps_m}^{i_n} \tau^{\eps_n})^{-1} 
h(\sigma_\eps) \tau^{-\eps} \lambda_{-\eps_1}^{i_1} \tau^{\eps_1} \cdots \lambda_{-\eps_m}^{i_n} \tau^{\eps_n} G = \\ 
\tau^{-\eps} \lambda_{-\eps_1}^{i_1} \tau^{\eps_1} \cdots \lambda_{-\eps_m}^{i_n} \tau^{\eps_n} \cdot 
\chi h(\iota_{\eps_n}, -\eps_n, k_{n-1}, -\eps_{n-1}, \dots, i_{\eps_1}, -\eps_1, \eps, \iota_{-\eps}; \sigma_\eps) \chi^{-1} G = \\ 
v(\iota_\eps, -\eps, i_1, \eps_1, \dots, i_n, \eps_n).
\end{multline*}
Consequently, $h(\sigma_\eps) \in \bar{\Theta}_{-\eps}$. \\ 
(iv) Note that, the element 
$\gamma = \tau^{-e_m} (\lambda_{-e_m}^{j_m})^{-1} \cdots \tau^{-e_1} (\lambda_{-e_1}^{j_1})^{-1} 
\lambda_{-\eps_1}^{i_1} \tau^{\eps_1} \cdots \lambda_{-\eps_n}^{i_n} \tau^{\eps_n}$ belongs to $T_{-e_m}^\dagger$ because of the condition 
$(i_1, \eps_1 \dots, i_m, \eps_m) \not= (j_1, e_1 \dots, j_m, e_m)$. It follows from Lemma \ref{HNNgeneration} (iii) that 
$\gamma = \tau^{-e_m} \lambda_{-l_1}^{k_1} \tau^{l_1} \lambda_{-l_2}^{k_2} \tau^{l_2} \cdots \lambda_{-l_s}^{k_s} \tau^{l_s} h$, where 
$h \in G$ and where $k_t \in I_{-l_t}$. Then, 
\begin{multline*}
h(j_1, e_1 \dots, j_m, e_m; \delta_m) \in \Lambda_{v(i_1, \eps_1 \dots, i_n, \eps_n)} \ \iff \\ 
\lambda_{-e_1}^{j_1} \tau^{e_1} \cdots \lambda_{-e_m}^{j_m} \tau^{e_m} h(\delta_m) 
\tau^{-e_m} (\lambda_{-e_m}^{j_m})^{-1} \cdots \tau^{-e_1} (\lambda_{-e_1}^{j_1})^{-1} \in \Lambda_{v(i_1, \eps_1 \dots, i_n, \eps_n)} \ \iff \\ 
h(\delta_m) \in \tau^{-e_m} (\lambda_{-e_m}^{j_m})^{-1} \cdots \tau^{-e_1} (\lambda_{-e_1}^{j_1})^{-1} 
\Lambda_{v(i_1, \eps_1 \dots, i_n, \eps_n)} \lambda_{-e_1}^{j_1} \tau^{e_1} \cdots \lambda_{-e_m}^{j_m} \tau^{e_m} \ \iff \\ 
h(\delta_m) \in \Lambda_{\tau^{-e_m} (\lambda_{-e_m}^{j_m})^{-1} \cdots \tau^{-e_1} (\lambda_{-e_1}^{j_1})^{-1}  v(i_1, \eps_1 \dots, i_n, \eps_n)} 
\ \iff \\ 
h(\delta_m) \in \Lambda_{\tau^{-e_m} (\lambda_{-e_m}^{j_m})^{-1} \cdots \tau^{-e_1} (\lambda_{-e_1}^{j_1})^{-1}  
\lambda_{-\eps_1}^{i_1} \tau^{\eps_1} \cdots \lambda_{-\eps_n}^{i_n} \tau^{\eps_n} G} \ \iff \\ 
h(\delta_m) \in \Lambda_{\tau^{-e_m} \lambda_{-l_1}^{k_1} \tau^{l_1} \lambda_{-l_2}^{k_2} \tau^{l_2} \cdots \lambda_{-l_s}^{k_s} \tau^{l_s} h  G} 
\ \iff \\ 
h(\delta_m) \in \Lambda_{v(\iota_{e_m}, -e_m, k_1, l_1, \dots, k_s, l_s)}.
\end{multline*}
Last line holds according to (iii). The inclusion $h(i_1, \eps_1 \dots, i_n, \eps_n; \sigma_n) \in \Lambda_{v(j_1, e_1 \dots, j_m, e_m)}$ 
is proven analogously. \\ 
(v) Every vertex of $\Lambda_{(\bar{\Theta}_\eps)}$ is of the form $v(\iota_{-\eps}, \eps, j_1, e_1 ,\dots, j_m, e_m)$, so if tuples 
$(i_1, \eps_1 \dots, i_n, \eps_n)$ and $(\iota_{-\eps}, \eps, j_1, e_1 ,\dots, j_m, e_m)$ satisfy the assumptions of (iv), then 
$h(i_1, \eps_1 \dots, i_n, \eps_n; \sigma_n) \in \Lambda_{(\bar{\Theta}_\eps)}$. By (i), 
$h(\iota_{-\eps}, \eps, j_1, e_1 ,\dots, j_m, e_m; \sigma_m) \notin \Lambda_{(\bar{\Theta}_\eps)}$, and the statement follows. 
\end{proof} 

\begin{prop} \label{LambdaK} 
For a group $\Lambda = \Lambda[I_{-1}, I_1, \iota_{-1}, \iota_1; \Sigma_{-1}, \Sigma_1]$, the following hold
\begin{multline*} 
\text{   (i)  } \Lambda_{(\bar{\Theta}_\eps)} = \langle \ \{ \ h(\sigma_{-\eps}) \ | \ \sigma_{-\eps} \in \Gamma_{-\eps} \ \} \ \cup \\ 
\{ \ h(i_1, \eps_1, \dots, i_m, \eps_m; \sigma_m) \ | \ m \geq 1, \ h(i_1, \eps_1, \dots, i_m, \eps_m; \sigma_m) \in H_{-\eps}, \text{ and } 
(i_1, \eps_1) \not= (\iota_{-\eps}, \eps) \ \} \ \rangle; 
\end{multline*} 
(ii) $K_\eps = \langle \{ \ h(\iota_\eps, -\eps; \sigma_{-\eps}) \ | \ \sigma_{-\eps} \in \Gamma_{-\eps} \ \} \sqcup 
\{ \ h(\iota_\eps, -\eps, i_1, \eps_1, \dots, i_n, \eps_n; \sigma_n) \ | \ n \geq 1, \sigma_n \in \Gamma_{\eps_n} \ \} \rangle.$ \\ 
(iii) $\ker \Lambda = \{ 1 \}$.
\end{prop}

\begin{proof}
(i) Denote the group on the right-hand-side by $\Delta$. The inclusion $\Delta < \Lambda_{(\bar{\Theta}_\eps)}$ follows from 
Lemma \ref{HNNaction} (iii) and (v). 
Take an element $h \in \Lambda_{(\bar{\Theta}_\eps)}$. From Proposition \ref{KepsTheta} (iv), it follows that $h \in H_{-\eps}$, 
so if $h = h(\sigma)$, then $\sigma \in \Gamma_{-\eps}$. If $h \not= h(\sigma)$, Lemma \ref{HNNgeneration} (ii) can be applied for 
$h^{-1} \in H_{-\eps}$. It follows that 
$$h =   \prod_{k=1}^{m} h(i_1^k, \eps_{k, 1}, \dots, i_{n_k}^k, \eps_{k, n_k}; \sigma_k) \cdot h(\sigma_{-\eps}),$$
where $m \geq 0, \ \sigma_k \in \Gamma_{\eps_{k, n_k}}$, $n_1 \geq n_2 \geq \dots \geq n_m \geq 1$, and 
$\sigma_{-\eps} \in \Gamma_{-\eps}$. Assume $h(i_1^l, \eps_{l, 1}, \dots, i_{n_l}^l, \eps_{l, n_l}; \sigma_l) \notin \Delta$ for some 
$1\leq l \leq m$ and that $l$ is the biggest number with this property. Then, it is clear that $i_1^l = \iota_{-\eps}$ and $\eps_{l,1} = \eps$. 
Also, $\sigma_l \in \Gamma_{\eps_{l, n_l}}$ is not the identity, so there 
exist two different elements $\kappa, \rho \in I_{-1} \sqcup I_1$, such that $\sigma_l(\kappa) = \rho$. Let $h$ act on 
$$v = v(i_1^l, \eps_{l, 1}, \dots, i_{n_l}^l, \eps_{l, n_l}, \kappa, \eps_{l, n_l}, \alpha_1, e_1, \dots, \alpha_{n_1}, e_{n_1}),$$ 
where $\alpha$'s and $e$'s are arbitrary and allowed. The terms $h(\sigma_{-\eps})$ and 
$\prod_{k={l+1}}^{m} h(i_1^k, \eps_{k, 1}, \dots, i_{n_k}^k, \eps_{k, n_k}; \sigma_k)$ leave $v$ fixed. From the final condition of Lemma 
\ref{HNNgeneration} (ii) and from Lemma \ref{HNNaction} (iv), it follows that the terms with length equal to $n_l$ also leave $v$ fixed. 
Finally, from Lemma \ref{HNNaction} (i), it follows that the remaining terms act on $v$ by eventually changing only the $\alpha$'s. 
Therefore we conclude that 
\begin{multline*} 
h v(i_1^l, \eps_{l, 1}, \dots, i_{n_l}^l, \eps_{l, n_l}, \kappa, \eps_{l, n_l}, \alpha_1, e_1, \dots, \alpha_{n_1}, e_{n_1}) = \\ 
v(i_1^l, \eps_{l, 1}, \dots, i_{n_l}^l, \eps_{l, n_l}, \rho, \eps_{l, n_l}, \beta_1, e_1, \dots, \beta_{n_1}, e_{n_1})
\end{multline*}
for some $\beta$'s. This shows that $h \notin \Lambda_{(\bar{\Theta}_\eps)}$, a contradiction that proves (i). \\ 
(ii) From Proposition \ref{KepsTheta} (iii), it follows that 
$$K_\eps =  \tau^{-\eps} K_\eps(\tau^{-\eps}) \tau^\eps = \tau^{-\eps} \Lambda_{(\bar{\Theta}_\eps)} \tau^\eps = 
\theta^\eps (\Lambda_{(\bar{\Theta}_\eps)}).$$ 
The assertion follows from relation (R7) and Lemma \ref{HNNgeneration} (i). \\
(iii) is obvious.
\end{proof} 

Now, we want to explore the structure of the quasi-kernels of $\Lambda = \Lambda[I_{-1}, I_1, \iota_{-1}, \iota_1; \Sigma_{-1}, \Sigma_1]$, 
in particular, that of $\Lambda_{(\bar{\Theta}_\eps)}$. 
\par 
First, we note that from Proposition \ref{LambdaK} (ii) and relation (R6), it follows that, for $i \in I_\eps$, 
\begin{multline*} \label{Lconjugations}
\lambda_\eps^i    \tau^{-\eps}   \Lambda_{(\bar{\Theta}_\eps)}   \tau^\eps  (\lambda_\eps^i)^{-1} = 
\lambda_\eps^i K_\eps (\lambda_\eps^i)^{-1} = \\ 
\langle \ \{ \ h(i,  -\eps, i_1, \eps_1, \dots, i_m, \eps_m; \sigma_m) \ | \ m \geq 0, \ h(i,  -\eps, i_1, \eps_1, \dots, i_m, \eps_m; \sigma_m) \in 
\underline{H} \ \} \ \rangle.
\end{multline*} 
It is clear that 
\begin{multline*} 
\Lambda_{(\bar{\Theta}_\eps)} = \\ 
\langle \ \{ \ h(\sigma_{-\eps}) \ | \ \sigma_{-\eps} \in \Gamma_{-\eps} \ \} \ \cup 
\ \underset{i \in I_\eps}{\cup} \lambda_\eps^i    \tau^{-\eps}   \Lambda_{(\bar{\Theta}_\eps)}   \tau^\eps  (\lambda_\eps^i)^{-1} \ \cup \ 
\underset{i \in I'_{-\eps}}{\cup} \lambda_{-\eps}^i    \tau^\eps   \Lambda_{(\bar{\Theta}_{-\eps})}   \tau^{-\eps}  (\lambda_{-\eps}^i)^{-1} \ 
\rangle = \\ 
\langle \ \{ \ h(\sigma_{-\eps}) \ | \ \sigma_{-\eps} \in \Gamma_{-\eps} \ \} \ \cup \mathcal{K}(0, -\eps) \ \rangle.
\end{multline*}
In other words, 
$$\Lambda_{(\bar{\Theta}_\eps)} \cong  \mathcal{K}(0, -\eps) \rtimes \Gamma_{-\eps}.$$
This can be written "recursively" as 
\begin{equation}  \label{HNNrecursive}
K_\eps \cong  [\underset{\#(S'_{-\eps})}{\bigoplus} K_{-\eps} \oplus \underset{\#(S_\eps)}{\bigoplus} K_\eps ] \rtimes \Gamma_{-\eps}.
\end{equation}
This is in a sense a "wreath product" representation. 
\par
Let's denote
\begin{equation*}
\mathcal{H}_\eps(0) = \langle \ \{ \  h(\sigma_{-\eps}) \ | \ \sigma_{-\eps} \in \Gamma_{-\eps} \ \} \ \rangle. 
\end{equation*}
For $n \geq 1$, let 
\begin{multline*}
\mathcal{H}_\eps(n) = \langle \ \{ \ h(i_1, \eps_1, \dots, i_n, \eps_n; \sigma_n) \ | \ h(i_1, \eps_1, \dots, i_n, \eps_n; \sigma_n) \in H_{-\eps} \text{ and } (i_1, \eps_1) \not= (\iota_{-\eps}, \eps) \ \} \ \rangle. 
\end{multline*}
Note that, each $\mathcal{H}_\eps(n)$ is isomorphic to a direct sum of copies of $\Gamma_1$ and $\Gamma_{-1}$. \\ 
Let's also denote 
$$\mathcal{H}_\eps[n] \ = \ \langle \ \mathcal{H}_\eps(0) \cup \mathcal{H}_\eps(1) \cup \dots \cup \mathcal{H}_\eps(n) \ \rangle.$$ 
Relation (R3) implies that $\mathcal{H}_\eps(n) \vartriangleleft \mathcal{H}_\eps[n]$ and that there is an extension
\begin{equation} \label{HNNextension}
\{ 1 \} \longrightarrow \mathcal{H}_\eps(n) \longrightarrow \mathcal{H}_\eps[n] \longrightarrow 
\mathcal{H}_\eps[n-1] \longrightarrow \{ 1 \}.
\end{equation}
The natural embeddings $\mathcal{H}_\eps[m] \hookrightarrow \mathcal{H}_\eps[n]$ give a representation of $\Lambda_{(\bar{\Theta}_\eps)}$ 
as a direct limit of groups 
\begin{equation} \label{Hlim}
\Lambda_{(\bar{\Theta}_\eps)} \ = \ \underset{\underset{n}{\longrightarrow}}{\lim} \mathcal{H}_\eps[n].
\end{equation}

\begin{lem} \label{Kepsamenable}
$K_{-1}$ is amenable if and only if $K_1$ is amenable, if and only if $\Gamma_{-1}$ and $\Gamma_1$ are both amenable, and if and only if 
$\Sigma_{-1}$ and $\Sigma_1$ are both amenable. 
\end{lem}

\begin{proof} 
Assume that $\Gamma_\eps$ is not amenable for some $\eps = \pm 1$. Then, by equation (\ref{HNNrecursive}), it follows that $K_{-\eps}$ is not amenable, so equation (\ref{HNNrecursive}), applied once more, gives the nonamenability of $K_\eps$. 
\par
Conversely, assume that $\Gamma_{-1}$ and $\Gamma_1$ are both amenable. Then, $\mathcal{H}_\eps(n)$ is amenable as a direct sum of copies of 
$\Gamma_{-1}$ and $\Gamma_1$. Also, $\mathcal{H}_\eps[0] = \mathcal{H}_\eps(0) \cong \Gamma_{-\eps}$ is amenable for $\eps = \pm 1$. 
Therefore, an easy induction based on the extension (\ref{HNNextension}), gives the amenability of $\mathcal{H}_\eps[n]$ for each $\eps = \pm 1$ 
and each $n \geq 0$. Finally, the direct limit representation (\ref{Hlim}) of $\Lambda_{(\bar{\Theta}_\eps)}$ implies the amenability of 
$\Lambda_{(\bar{\Theta}_\eps)}$ for $\eps = \pm 1$. Since $K_\eps =  \tau^{-\eps} \Lambda_{(\bar{\Theta}_\eps)} \tau^\eps$, 
it is amenable too. 
\end{proof}

\subsection{Group-Theoretic Structure} 

We give a result about the structure of the groups. 

\begin{thm} \label{HNNstructure}
Let's denote $\Lambda = \Lambda[I_{-1}, I_1, \iota_{-1}, \iota_1; \Sigma_{-1}, \Sigma_1]$. Let's assume that: \\ 
\text{(i)} $\Sigma_{-1}$ and $\Sigma_1$ are $2$-transitive, that is, all stabilizers 
$(\Sigma_\eps)_{i_\eps}$ are transitive on the sets $I_\eps \setminus \{ i_\eps \}$ for all $i_\eps \in I_\eps$ and $\eps = \pm 1$; \\ 
\text{(ii)} For each $\eps = \pm 1$, either $\Sigma_\eps = \langle (\Sigma_\eps)_{i_\eps} \ | \ i_\eps \in I_\eps \rangle$ or 
$\Sigma_\eps = \text{Sym}(2)$. 
\par
Then, $\Lambda$ has a simple normal subgroup $\Xi$ for which there is a group extension 
$$1 \longrightarrow \Xi \longrightarrow \Lambda \overset{\eta}{\longrightarrow} 
(\Gamma / [\Gamma, \Gamma]) \wr_\mathbb{Z} \mathbb{Z} 
\longrightarrow 1,$$
where $\eta$ is defined on the generators by 
$$\eta(h(\sigma)) = ((\dots, 0, \dots, 0, ([\sigma],0), 0, \dots, 0, \dots), 0),\ \ \ \eta(\tau) = ((\dots, 0, \dots), 1), \ \ \text{ and }$$ 
$$\eta(h(i_1, \eps_1, \dots, i_n, \eps_n; \sigma_n)) = ((\dots, 0, \dots, 0, ([\sigma_n], \eps_1 + \dots + \eps_n), 0, \dots, 0, \dots), 0).$$
Here, $[\sigma]$ denotes the image of the permutation $\sigma \in \Gamma$ in $\Gamma / [\Gamma, \Gamma]$.
\end{thm}

\begin{proof}
We note that it follows from relations (R7), (R8), and (R9) that the action of $\theta$ on an element $h(i_1, \eps_1, \dots, i_n, \eps_n; \sigma_n)$ 
is consistent with the definition of $\eta$ and the multiplication in the wreath product, that is, 
\begin{multline*} 
\eta(\theta(h(i_1, \eps_1, \dots, i_n, \eps_n; \sigma_n))) = \eta(\tau^{-1} h(i_1, \eps_1, \dots, i_n, \eps_n; \sigma_n) \tau) \ = \\  
((\dots, 0, \dots, 0, ([\sigma_n], \eps_1 + \dots + \eps_n - 1), 0, \dots, 0, \dots), 0). 
\end{multline*} 
It is easy to see that, since the commutant is in the kernel, the homomorphism 
$\eta : G \to (\Gamma / [\Gamma, \Gamma]) \wr_\mathbb{Z} \mathbb{Z}$ is well defined by 
$$\eta(g) \ = \ ((\dots, (\prod_{\eps_1 + \dots + \eps_n = m} [ \sigma_n ] , m),  \dots), 0),$$ 
where the products are taken over all the factors $h(i_1, \eps_1, \dots, i_n, \eps_n; \sigma_n)$ of $g$. 
These two observations together with the universal property of the HNN-extensions (Remark \ref{HNNuniversal}) enable us to extend 
$\eta$ to the entire group $\Lambda$. 
\par
Now, notice that if $\lambda = g_1 \tau^{\eps_1} g_2 \tau^{\eps_2} g_3 \tau^{\eps_3} \cdots g_n \tau^{\eps_n} g_{n+1} \ \in \ \Xi$, then 
$\eps_1 + \dots + \eps_n = 0$. 
Thus 
$$\lambda \ = \ g_1 (\tau^{\eps_1} g_2 \tau^{-\eps_1}) (\tau^{\eps_1 + \eps_2} g_3 \tau^{-\eps_1 - \eps_2}) \cdots 
(\tau^{\eps_1 + \eps_2 + \dots + \eps_{n-1}} g_n \tau^{-\eps_1 - \eps_2 - \dots - \eps_{n-1}}) g_{n+1} $$
can be represented as products of $\tau$-conjugates of elements from $G$. 
\par 
Using Lemma \ref{HNNgeneration} (ii), we see that every $\lambda = \tau^n g \tau^{-n}$ can be written as a product of elements of the form 
$\tau^n h(\sigma) \tau^{-n}$ and $\tau^n h(i_1, \eps_1 \dots, i_m, \eps_m; \sigma_m) \tau^{-n}$. The second element equals either 
$\tau^{n-m} h(\sigma_m) \tau^{m-n}$ or $h(j_1, \eps'_1, \dots, j_k, \eps'_k; \sigma_m)$ for some $j_p$'s and $\eps'_p$'s. 
Therefore, it is easy to see that $\Xi$ is generated by the following set 
\begin{multline} \label{Xigeneration}
\{ h(i_1, \eps_1, \dots, i_n, \eps_n; \sigma_n) h(i'_1, \eps_1, \dots, i'_n, \eps_n; \sigma_n^{-1}) \ | \ 
\eps_k = \pm 1, \ i_k, i'_k \in I_{-\eps_k}\ n \in \mathbb{N},\ \sigma_n \in \Gamma_{\eps_n} \} \\ 
\cup \ \{ i, \eps, i_0, -\eps, i_1, \eps, i_2, \eps_2, \dots, i_n, \eps_n; \sigma_n)  
h(\bar{i}, \eps, i'_2, \eps_2, \dots, i_n, \eps_n; \sigma_n^{-1}) \ | \\ 
n \geq 2,\ \eps, \eps_k = \pm 1,\ i_0 \in I_\eps,\ i'_2 \in I_{-\eps_2},\ i_k \in I_{-\eps_k}, \ i, \bar{i} \in I_{-\eps} \} \ \cup \\ 
\{ h(\sigma_\eps) h(i_\eps, -\eps, i_{-\eps}, \eps; \sigma_\eps^{-1}) \ | \ \sigma_\eps \in \Gamma_\eps,\ i_{-\eps} \in I'_\eps, \ 
i_\eps \in I_{-\eps}, \ \eps  = \pm 1 \} \ \cup \\
\{ h(i_1, \eps_1, \dots, i_m, \eps_m, i, \eps, j, -\eps, j_1, \eps'_1, \dots, j_n, \eps'_n; \sigma) \cdot \\
h(i_1, \eps_1, \dots, i_m, \eps_m, j', -\eps, i', \eps, j_1, \eps'_1, \dots, j_n, \eps'_n; \sigma^{-1}) \ | \\ 
i_k \in I_{-\eps_k}, \ j_k \in I_{-\eps'_k}, \ i,i', \in I_{-\eps}, j, j' \in I_\eps, \ \sigma \in \Gamma_{\eps'_n}, \ 
\eps, \eps_k, \eps'_k = \pm 1, \ m, n \in \mathbb{N}_0 \} \ \cup \\ 
\{ \tau^{\eps n} h(\sigma_{-\eps}) \tau^{-\eps n} 
h(\underbrace{\iota_{-\eps}, \eps, \dots,\iota_{-\eps}, \eps }_{n\text{ times}};\sigma_{-\eps}^{-1}) \ | 
\ \sigma_{-\eps} \in \Gamma_{-\eps}, \eps = \pm 1, \ n \in \mathbb{N} \} \ \cup \\ 
\{ \tau^{\eps n} h(\sigma_{-\eps}) \tau^{-\eps n} \ | 
\ n \in \mathbb{N},\ \sigma_{-\eps} \in \Gamma_{-\eps} \cap [\Gamma, \Gamma], \ \eps = \pm 1 \} \ \cup  \ 
\{  h(\sigma) \ | \ \sigma \in [\Gamma, \Gamma] \}
\end{multline}

Take any element $a \in \Xi \backslash \{ 1 \}$. It remains to show that $\langle \langle a \rangle \rangle_\Xi = \Xi$. 	
For big enough $n$, we can find $i_k$'s so that the element $h(i_1, \eps_1, \dots, i_n, \eps_n; \sigma_n)$ does not commute with $a$. So, for any 
$i'_k$'s (different from $i_k$'s) we have 
$$v \ \equiv \ h(i_1, \eps_1, \dots, i_n, \eps_n; \sigma_n) h(i'_1, \eps_1, \dots, i'_n, \eps_n; \sigma_n^{-1}) \ \in \ \Xi,$$
and by relations (R3), (R6), (R8), and (R9), it follows 
\begin{multline*} 
\langle \langle a \rangle \rangle_\Xi \ni b \equiv a v a^{-1} v \ = \\ 
\cdot h(i'_1, \eps_1'', \dots, i'_l, \eps''_l; \sigma_n) h(i_1, \eps_1, \dots, i_n, \eps_n; \sigma_n^{-1}) = \\ 
h(j_1, \eps'_1, \dots, j_m, \eps'_m, j_{m+1}, \eps_n; \sigma_n) h(j'_1, e''_1, \dots, j'_d, e''_d, j'_{d+1}, \eps_n; \sigma_n^{-1}) \cdot \\ 
\cdot h(i_1, \eps_1, \dots, i_n, \eps_n; \sigma_n) h(i'_1, \eps_1, \dots, i'_n, \eps_n; \sigma_n^{-1}). 
\end{multline*} 
for some $j_k$'s, $j'_k$'s, $i'_k$'s, $\eps'_k$'s, $e''_k$'s, and $\eps''_k$'s. \\ 
Now, it is clear that we can find big enough $s$ and appropriate $e_k$'s, $p_k$'s, $l_k$'s, and $q_k$'s, so that 
$h(q_1, l_1, \dots, q_s, l_s; \sigma^{-1})$ commutes with $b$ and $h(p_1, e_1, \dots, p_s, e_s; \sigma)$ does not. Then, 
\begin{multline*} 
\langle \langle a \rangle \rangle_\Xi \ni b' \equiv \\ 
b h(p_1, e_1, \dots, p_s, e_s; \sigma) h(q_1, l_1, \dots, q_s, l_s; \sigma^{-1}) b^{-1} 
h(q_1, l_1, \dots, q_s, l_s; \sigma) h(p_1, e_1, \dots, p_s, e_s; \sigma^{-1}) = \\ 
h(p'_1, e_1, \dots, p'_s, e_s; \sigma) h(p_1, e_1, \dots, p_s, e_s; \sigma^{-1}) \not= 1
\end{multline*}
for some $p'_k$'s. We can adjust $s$ to be a big enough and adjust the 'tail' of $(p_1, e_1, \dots, p_s, e_s)$ so that 
$e_1 + \dots + e_n = 0$. Since the tuples 
$(p_1, e_1, \dots, p_s, e_s)$ and $(p'_1, e_1, \dots, p'_s, e_s)$ are different, it follows from Lemma 
\ref{HNNgeneration} (i) and from the assumption $\eps_1 + \dots + \eps_n = 0$ that 
$$\beta b' \beta^{-1} =  h(p''_1, e'''_1, \dots, p''_k, e'''_k, p'', e_s; \sigma) h(\sigma^{-1}) \ \in \ \langle \langle a \rangle \rangle_\Xi$$
for some $k \in \mathbb{N}$, $p''_l$'s, and $e'''_l$'s, where 
\begin{multline*} 
\Xi \ \ni \ \beta \ = \ \tau^{-e_s} (\lambda_{-e_s}^{p_s})^{-1} \cdots \tau^{-e_1} (\lambda_{-e_s}^{p_1})^{-1} \cdot \\ 
\cdot \prod_{e_k = -1} h(\rho_1^k, w_1^k, \dots, \rho_{t_k}^k, w_{t_k}^k, w, 1; \mu_{-e_k}^{p_k}) \cdot 
\prod_{e_k = 1} h(\bar{\rho}_1^k, \bar{w}_1^k, \dots, \bar{\rho}_{t'_k}^k, \bar{w}_{t'_k}^k, \bar{w}, -1; \mu_{-e_k}^{p_k}), 
\end{multline*} 
and where the last two factors are chosen appropriately. 
This argument does not depend on the 'tail' of $(p_1, e_1, \dots, p_s, e_s)$, therefore we can take $e_s$ 
to be either $1$ or $-1$. 
\par
We conclude that the following are elements of $\langle \langle a \rangle \rangle_\Xi$: 
$$c = h(\sigma_1) h(\iota_1, -1, p_1, e_1, \dots, p_k, e_k, p, 1; \sigma_1^{-1}) \text{ and } 
d = h(\sigma_{-1}) h(\iota_{-1}, 1, q_1, l_1, \dots, q_k, l_k, q, -1; \sigma_{-1}^{-1})$$
for any big enough even number $k$, for any $\sigma_1 \in \Gamma_1$ and $\sigma_{-1} \in \Gamma_{-1}$, and for some $p_m$'s, $q_m$'s, 
$e_m$'s, and $l_m$'s. 
\par
We claim that, in the tuples $(\iota_1, -1, p_1, e_1, \dots, p_k, e_k, p, 1)$ and 
$(\iota_{-1}, 1, q_1, l_1, \dots, q_k, l_k, q, -1)$, the indices $p$, $q$, $p_t$'s, and $q_t$'s can be chosen arbitrary. To see this, consider 
\begin{equation*}
\Xi \ \ni \ f \ = \ 
h(\iota_1, -1, p_1, e_1, \dots, p_t, e_t; \omega_t) h(q_0, -1, q_1, o_1, \dots, q_r, o_r, q, e_t; \omega_t^{-1}), 
\end{equation*} 
where $q_0 \not= \iota_1$ and the second factor is chosen appropriately. Then, by relation (R3), 
$$f c f^{-1} \ = \ h(\sigma_1) h(\iota_1, -1, p_1, e_1, \dots, \omega_t(p_{t+1}), \dots, p_k, e_k, p, 1; \sigma_1^{-1}) \ \in \ 
\langle \langle a \rangle \rangle_\Xi.$$ 
Because of the transitivity and $2$-transitivity of $\Sigma_{-1}$ and $\Sigma_1$, the claim is proven. 
The element $d$ can be manipulated similarly. 
\par 
Now, consider 
$$\Xi \ \ni \ s \ = \ 
h(\iota_{-1}, 1, i_2, \eps_2, \dots, i_t, \eps_t; \omega_t) h(\iota_1, -1, q'_1, o'_1, \dots, q'_r, o'_r, q', e_t; \omega_t^{-1})$$
for an appropriate choice of $q'_l$'s and $p_l$'s so it commutes with $h(\iota_1, -1, p_1, e_1, \dots, p_k, e_k, p, 1; \sigma_1^{-1})$. 
Therefore, 
$$s c s^{-1} c^{-1} \ = \ h(\iota_{-1}, 1, i_2, \eps_2, \dots, i_t, \eps_t; \omega_t) 
h(\sigma_1(\iota_{-1}), 1, i_2, \eps_2, \dots, i_t, \eps_t; \omega_t^{-1}) 
\ \in \ \langle \langle a \rangle \rangle_\Xi,$$ 
so by the transitivity of the group $\Sigma_{-1}$, we see that every element of the form
$$h(\iota_{-1}, 1, i_2, \eps_2, \dots, i_t, \eps_t; \omega_t) h(i_1, 1, i_2, \eps_2, \dots, i_t, \eps_t; \omega_t^{-1})$$ 
belongs to $\langle \langle a \rangle \rangle_\Xi$. Products of such elements yield 
$$h(i'_1, 1, i_2, \eps_2, \dots, i_t, \eps_t; \omega_t) h(i_1, 1, i_2, \eps_2, \dots, i_t, \eps_t; \omega_t^{-1}) \ \in \ 
\langle \langle a \rangle \rangle_\Xi.$$ 
By making the same argument that uses transitivity and $2$-transitivity, we see that we can change the $i_l$ indices of the first factor, so we infer that 
the first set of (\ref{Xigeneration}) belongs to $\langle \langle a \rangle \rangle_\Xi$. 
\par
Consider, for $n \geq 2$, an even number $k$ an appropriate $h(j_1, \eps'_1, \dots, j_k, \eps'_k; \sigma)$ that commutes with 
$h(i_1, \eps_1, i_2, \eps_2, \dots, i_n, \eps_n; \sigma_n)$ and with $h(\iota_{-\eps_1}, \eps_1, i_2, \eps_2, \dots, i_n, \eps_n; \sigma_n^{-1})$ 
and for which 
$$\delta' \equiv \tau^{\eps_1} h(\sigma) \tau^{-\eps_1} h(j_1, \eps'_1, \dots, j_k, \eps'_k; \sigma^{-1})$$ 
belongs to $\Xi$.
Then, 
\begin{multline*} 
\delta' h(i_1, \eps_1, i_2, \eps_2, \dots, i_n, \eps_n; \sigma_n) 
h(\iota_{-\eps_1}, \eps_1, i_2, \eps_2, \dots, i_n, \eps_n; \sigma_n^{-1}) (\delta')^{-1} \ = \\ 
h(\iota_{-\eps_1}, \eps_1, \sigma(\iota_{\eps_1}), -\eps_1, i_1, \eps_1, i_2, \eps_2, \dots, i_n, \eps_n; \sigma_n)  
h(\iota_{-\eps_1}, \eps_1, \sigma(i_2), \eps_2, \dots, i_n, \eps_n; \sigma_n^{-1}) \ \in \ \langle \langle a \rangle \rangle_\Xi.
\end{multline*}
Products of those elements with elements from the first set give all the elements from the second set of (\ref{Xigeneration}), so it is included in 
$\langle \langle a \rangle \rangle_\Xi$. 
\par
The third set of (\ref{Xigeneration}) belongs to $\langle \langle a \rangle \rangle_\Xi$ since its elements are products of $c$ and $d$ with 
elements from the second set. 
\par
A generic element of the fourth set of (\ref{Xigeneration}) can be written as 
\begin{multline} \label{4generic}
h(i_1, \eps_1, \dots, i_m, \eps_m, i, \eps, j, -\eps, \bar{i}, \eps, j_2, \eps'_2, \dots, j_n, \eps'_n; \sigma) \cdot \\ 
h(i_1, \eps_1, \dots, i_m, \eps_m, j', -\eps, i', \eps, \bar{i}, \eps, j_2, \eps'_2, \dots, j_n, \eps'_n; \sigma^{-1}),
\end{multline}
where we have assumed, without loss of generality, that $\eps'_1 = \eps$. We must show that this element belongs to 
$\langle \langle a \rangle \rangle_\Xi$. 
\par 
First, we start with the following element from the first set of (\ref{Xigeneration}) 
\begin{multline*}
 \langle \langle a \rangle \rangle_\Xi \ \ni \ z \ = \ 
h(i_1, \eps_1, \dots, i_m, \eps_m, i, \eps, \iota_{-\eps}, \eps, q, -\eps, j, -\eps, \bar{i}, \eps, j_2, \eps'_2, \dots, j_n, \eps'_n; \sigma) \cdot \\
h(i_1, \eps_1, \dots, i_m, \eps_m, i, \eps, \iota_{-\eps}, \eps, q, -\eps, \iota_\eps, -\eps, \bar{i}, \eps, j_2, \eps'_2, \dots, j_n, \eps'_n; \sigma^{-1}),
\end{multline*}
where $q \in I'_\eps$. 
\par
Next, using Lemma \ref{HNNgeneration} (i) and adopting the notations thereof, we define 
\begin{multline*}
\Xi \ \ni \ \gamma \ = \ \lambda_{-\eps_1}^{i_1} \tau^{\eps_1} \cdots \lambda_{-\eps_m}^{i_m} \tau^{\eps_m} \lambda_{-\eps}^i 
               \tau^{2\eps}  (\lambda_\eps^q)^{-1}  \tau^{-2\eps} (\lambda_{-\eps}^i)^{-1}
\cdot \tau^{-\eps_m} (\lambda_{-\eps_m}^{i_m})^{-1} \cdots \tau^{-\eps_1} (\lambda_{-\eps_1}^{i_1})^{-1} \cdot \\ 
\cdot h(r_1, e_1, \dots, r_{2l-1}, e_{2l-1}, \bar{r}_{-\eps}, \eps; \mu_\eps^q) 
\end{multline*}
for appropriate $r_k$'s and $e_k$'s satisfying $e_1 + \dots + e_{2l-1} + \eps \ = \ 0$ and for 
which the last factor commutes with everything in the next expressions. Then, 
$$\gamma z \gamma^{-1} \ = \ 
h(i_1, \eps_1, \dots, i_m, \eps_m, i, \eps, j, -\eps, \bar{i}, \eps, j_2, \eps'_2, \dots, j_n, \eps'_n; \sigma) 
\cdot \bar{h},$$
where 
\begin{multline*}
\bar{h} \ \equiv \ \gamma 
h(i_1, \eps_1, \dots, i_m, \eps_m, i, \eps, \iota_{-\eps}, \eps, q, -\eps, \iota_\eps, -\eps, \bar{i}, \eps, j_2, \eps'_2, \dots, j_n, \eps'_n; \sigma^{-1}) 
\gamma^{-1} \ = \\ 
\lambda_{-\eps_1}^{i_1} \tau^{\eps_1} \cdots \lambda_{-\eps_m}^{i_m} \tau^{\eps_m} \lambda_{-\eps}^i 
\lambda_{-\eps}^{\bar{i}} \tau^\eps \cdots \lambda_{-\eps'_n}^{j_n} \tau^{\eps'_n} h(\sigma^{-1}) \ \cdot \\ 
\tau^{-\eps'_n} (\lambda_{-\eps'_n}^{j_n})^{-1} \cdots \tau^{-\eps} (\lambda_{-\eps}^{\bar{i}})^{-1} (\lambda_{-\eps}^i)^{-1} 
\tau^{-\eps_m} (\lambda_{-\eps_m}^{i_m})^{-1} \cdots \tau^{-\eps_1} (\lambda_{-\eps_1}^{i_1})^{-1} = \\ 
\lambda_{-\eps_1}^{i_1} \tau^{\eps_1} \cdots \lambda_{-\eps_m}^{i_m} \tau^{\eps_m} \lambda_{-\eps}^i 
h(\bar{i}, \eps, j_2, \eps'_2, \dots, j_n, \eps'_n; \sigma^{-1}) 
(\lambda_{-\eps}^i)^{-1} \tau^{-\eps_m} (\lambda_{-\eps_m}^{i_m})^{-1} \cdots \tau^{-\eps_1} (\lambda_{-\eps_1}^{i_1})^{-1} = \\ 
\lambda_{-\eps_1}^{i_1} \tau^{\eps_1} \cdots \lambda_{-\eps_m}^{i_m} \tau^{\eps_m} 
h(\mu_{-\eps}^i(\bar{i}), \eps, j_2, \eps'_2, \dots, j_n, \eps'_n; \sigma^{-1}) 
\tau^{-\eps_m} (\lambda_{-\eps_m}^{i_m})^{-1} \cdots \tau^{-\eps_1} (\lambda_{-\eps_1}^{i_1})^{-1}. 
\end{multline*}
Likewise, we consider the following element from the first set of (\ref{Xigeneration}) 
\begin{multline*}
 \langle \langle a \rangle \rangle_\Xi \ \ni \ z' \ = \ 
h(i_1, \eps_1, \dots, i_m, \eps_m, j', -\eps, \iota_\eps, -\eps, p, \eps, \iota_{-\eps}, \eps, \mu_{-\eps}^i(\bar{i}), \eps, 
j_2, \eps'_2, \dots, j_n, \eps'_n; \sigma) \cdot \\
h(i_1, \eps_1, \dots, i_m, \eps_m, j', -\eps, \iota_\eps, -\eps, p, \eps, i', \eps, \mu_{-\eps}^i(\bar{i}), \eps, 
j_2, \eps'_2, \dots, j_n, \eps'_n; \sigma^{-1}), 
\end{multline*}
where $p \in I'_{-\eps}$ and define 
\begin{multline*}
\Xi \ \ni \ \gamma' \ = \ \lambda_{-\eps_1}^{i_1} \tau^{\eps_1} \cdots \lambda_{-\eps_m}^{i_m} \tau^{\eps_m} \lambda_\eps^{j'} 
               \tau^{-2\eps}  (\lambda_{-\eps}^p)^{-1}  \tau^{2\eps} (\lambda_\eps^{j'})^{-1}
\cdot \tau^{-\eps_m} (\lambda_{-\eps_m}^{i_m})^{-1} \cdots \tau^{-\eps_1} (\lambda_{-\eps_1}^{i_1})^{-1} \cdot \\ 
\cdot h(r'_1, e_1, \dots, r'_{2l-1}, e_{2l-1}, \bar{r}_{-\eps}, \eps; \mu_{-\eps}^p) 
\end{multline*}
for appropriate $r'_k$'s. Then, 
$$\gamma' z' (\gamma')^{-1} \ = \ \bar{\bar{h}} \cdot 
h(i_1, \eps_1, \dots, i_m, \eps_m, j', -\eps, i', \eps, \mu_{-\eps}^i(\bar{i}), \eps, j_2, \eps'_2, \dots, j_n, \eps'_n; \sigma^{-1}),$$
where 
\begin{multline*}
\bar{\bar{h}} \ \equiv \ \gamma' 
h(i_1, \eps_1, \dots, i_m, \eps_m, j', -\eps, \iota_\eps, -\eps, p, \eps, \iota_{-\eps}, \eps, \mu_{-\eps}^i(\bar{i}), \eps, 
j_2, \eps'_2, \dots, j_n, \eps'_n; \sigma) 
(\gamma')^{-1} \ = \\ 
\lambda_{-\eps_1}^{i_1} \tau^{\eps_1} \cdots \lambda_{-\eps_m}^{i_m} \tau^{\eps_m} 
\lambda_\eps^{j'} h(\mu_{-\eps}^i(\bar{i}), \eps, j_2, \eps'_2, \dots, j_n, \eps'_n; \sigma) (\lambda_\eps^{j'})^{-1} 
\tau^{-\eps_m} (\lambda_{-\eps_m}^{i_m})^{-1} \cdots \tau^{-\eps_1} (\lambda_{-\eps_1}^{i_1})^{-1} = \\ 
\lambda_{-\eps_1}^{i_1} \tau^{\eps_1} \cdots \lambda_{-\eps_m}^{i_m} \tau^{\eps_m} 
h(\mu_\eps^{j'}(\mu_{-\eps}^i(\bar{i})), \eps, j_2, \eps'_2, \dots, j_n, \eps'_n; \sigma) 
\tau^{-\eps_m} (\lambda_{-\eps_m}^{i_m})^{-1} \cdots \tau^{-\eps_1} (\lambda_{-\eps_1}^{i_1})^{-1} = \\ 
= \ (\bar{h})^{-1} 
\end{multline*}
since $\mu_\eps^{j'}(\mu_{-\eps}^i(\bar{i})) = \mu_{-\eps}^i(\bar{i})$, due to relation (R6) and $\mu_{-\eps}^i(\bar{i}) \in I_{-\eps}$. 
Finally, 
\begin{multline*}
\langle \langle a \rangle \rangle_\Xi \ \ni \ \gamma z \gamma^{-1} \cdot \gamma' z' (\gamma')^{-1} \ = \\ 
h(i_1, \eps_1, \dots, i_m, \eps_m, i, \eps, j, -\eps, \bar{i}, \eps, j_2, \eps'_2, \dots, j_n, \eps'_n; \sigma) \cdot \\ 
h(i_1, \eps_1, \dots, i_m, \eps_m, j', -\eps, i', \eps, \mu_{-\eps}^i(\bar{i}), \eps, j_2, \eps'_2, \dots, j_n, \eps'_n; \sigma^{-1}),
\end{multline*}
and after a multiplication with an element from the first set of (\ref{Xigeneration}), we get the element (\ref{4generic}). 
\par 
Therefore, the fourth set of  (\ref{Xigeneration}) is in $\langle \langle a \rangle \rangle_\Xi$. 
\par
Repeating almost verbatim the corresponding part of the proof of Proposition \ref{amalgamstructure} gives us that the seventh set of 
(\ref{Xigeneration}) belongs to $\langle \langle a \rangle \rangle_\Xi$. Note that if $\Sigma_\eps = \text{Sym}(2)$, then 
$[\Sigma_\eps, \Sigma_\eps]$ is the trivial group. 
\par
Next, we take numbers $m > n$ and 
$$\gamma'' \ = \ \tau^{\eps m} h(\sigma'_{-\eps}) \tau^{-\eps m} h(j_1, \eps, \dots, j_{m+1}, \eps, j, -\eps; (\sigma'_{-\eps})^{-1})  \ \in \ \Xi,$$ 
where $\sigma'_{-\eps} \in \Gamma_{-\eps}$, $j_{-\eps} \in I'_{-\eps}$, and $j_\eps \in I'_\eps$, with the relation 
$(\sigma'_{-\eps})^{-1}(\iota_\eps) = q$ for some $q \in I'_\eps$. \\ 
After that, we take the element of $\langle \langle a \rangle \rangle_\Xi$ (it is a product of elements from the second and fourth set) 
\begin{multline*} 
x \ \equiv \ h(\underbrace{\iota_{-\eps}, \eps, \dots, \iota_{-\eps}, \eps}_{m\text{ times}}, 
q, -\eps, \underbrace{\iota_\eps, -\eps, \dots, \iota_\eps, -\eps}_{m-n-1\text{ times}}; \sigma_{-\eps}) \cdot \\ 
\cdot h(\underbrace{\iota_{-\eps}, \eps, \dots, \iota_{-\eps}, \eps}_{m\text{ times}}, 
q, -\eps, \underbrace{\iota_\eps, -\eps, \dots, \iota_\eps, -\eps}_{m\text{ times}}, p, \eps, 
\underbrace{\iota_{-\eps}, \eps, \dots, \iota_{-\eps}, \eps}_{n-1\text{ times}}; \sigma_{-\eps}), 
\end{multline*} 
where $p \in I'_{-\eps}$. Then 
$$\gamma'' x (\gamma'')^{-1} \ = \\ \tau^{\eps n} h(\sigma_{-\eps}) \tau^{-\eps n} \cdot 
h(p, \eps, \underbrace{\iota_{-\eps}, \eps, \dots, \iota_{-\eps}, \eps}_{n-1\text{ times}}; \sigma_{-\eps}) \ 
\in \ \langle \langle a \rangle \rangle_\Xi.$$ 
Therefore, upon a multiplication by an element from the first set of (\ref{Xigeneration}), we infer that the fifth set of (\ref{Xigeneration}) 
belongs to $\langle \langle a \rangle \rangle_\Xi$.
\par
Finally, the argument from Lemma \ref{amalgamstructure} can be used for the sixth set of (\ref{Xigeneration}) the same way it was used for the 
seventh set. 
\par
This completes the proof.
\end{proof}

\begin{remark}
The example introduced in \cite[Section 5]{BIO} corresponds to the case $\Sigma_{-1} \cong \Sigma_1 \cong \text{Sym}(2)$. In the proof of 
\cite[Proposition 5.11]{BIO}, we introduced a group $\underline{\text{int}}\ \Gamma$ and wrongly stated that it is simple. It is clear from 
Theorem \ref{HNNstructure} that, rather, its subgroup $\Xi$ is simple. 
\end{remark}

\begin{remark}
Analogous arguments to the ones in Remark \ref{edgecommutants} can be given for the case of HNN-extensions, with the conclusion 
$$\Xi \ = \ \langle [\Lambda_e, \Lambda_e] \ | \ e \text{ is an edge of } T \rangle.$$ 
\end{remark}

\subsection{Analytic Structure}

\begin{lem} \label{HNNgeneraltype}
The group $\Lambda = \Lambda[I_{-1}, I_1, \iota_{-1}, \iota_1; \Sigma_{-1}, \Sigma_1]$ is a non-ascending HNN-extension and its action on its 
Bass-Serre tree is minimal and of general type. 
\end{lem}

\begin{proof} 
Since the action is transitive, it is minimal. Since $H \ \not= \ G \ \not= \ \theta(H)$, it follows that $\Lambda$ is nondegenerate and non-ascending. 
The result now follows from \cite[Proposition 20]{Harpepreaux}. 
\end{proof}

\begin{thm} \label{HNNC*simple}
The HNN-extension $\Lambda = \Lambda[I_{-1}, I_1, \iota_{-1}, \iota_1; \Sigma_{-1}, \Sigma_1]$ has a unique trace. It is $C^*$-simple if and only if 
either one of the groups $\Sigma_{-1}$ and $\Sigma_1$ is non-amenable.
\end{thm}

\begin{proof}
Lemma \ref{HNNgeneraltype} enables us to apply \cite[Theorem 4.19]{BIO} to conclude that $\Lambda$ has the unique trace property since 
$\ker \Lambda$ is trivial. It also enables us to apply \cite[Theorem 4.20]{BIO} to conclude that $\Lambda$ is $C^*$-simple if and only if $K_{-1}$ and 
$K_1$ are non-amenable, which, by Lemma \ref{Kepsamenable}, is equivalent to the requirement that either one of the groups $\Sigma_{-1}$ and $\Sigma_1$ 
to be non-amenable. 
\end{proof}

Finally, we prove 

\begin{thm} 
The HNN-extension $\Lambda = \Lambda[\Sigma_{-1}, \Sigma_1]$ in not inner amenable. 
\end{thm}

\begin{proof}
Lemma \ref{HNNgeneraltype} allows us to apply Proposition \ref{fledged}, so we need to show that the action of 
$\Lambda = \Lambda[I_{-1}, I_1, \iota_{-1}, \iota_1; \Sigma_{-1}, \Sigma_1]$ on its Bass-Serre is finitely fledged. 
\par 
For this, take any	 elliptic element $g \in \Lambda \setminus \{ 1 \}$. Since $g$ fixes some vertex, it is a conjugate of an element of $G$. 
The finite fledgedness property is conjugation invariant, so we can assume $g \in G \setminus \{ 1 \}$. 
\par 
From Lemma \ref{HNNgeneration} (ii), we can write $g = h(\sigma) h_{-1} h_1$, where $\sigma \in \Gamma$, 
$$h_{-1} = \prod_{k=1}^{m} h(i_1^k, -1, i_2^k, \eps_{k, 2}, \dots, i_{n_k}^k, \eps_{k, n_k}; \sigma_k), \ \ 
h_1 = \prod_{l=m+1}^{r} h(i_1^l, 1, i_2^l, \eps_{l, 2}, \dots, i_{n_l}^l, \eps_{l, n_l}; \theta_l),$$ 
$r \geq m \geq 0, \ \sigma_k \in \Gamma_{\eps_{k, n_k}},\ \theta_l \in \Gamma_{\eps_{l, n_l}}$, and $i_z^p \in I'_{\eps_{p, z}}$. 
We also require $0 \leq n_1 \leq \dots \leq n_m $ and $0 \leq n_{m+1} \leq \dots \leq n_r$. 
\par
Let's assume that $g$ fixes a vertex $v = v(i_1, \eps_1, \dots, i_n, \eps_n)$, where $n \geq \max \{ n_m, n_r \} + 1$, and take 
$w = v(i_1, \eps_1, \dots, i_n, \eps_n, \dots, i_{n+d}, \eps_{n+d})$ for any $d \geq 1$. We note that, $h_{-\eps_1}$ fixes 
$w$ and $h(\sigma) h_{\eps_1}$ modifies only indices with numbers no greater than $\{ n_m, n_r \} + 1 \leq n$. Therefore, 
$$h(\sigma) h_{\eps_1} v = v(i'_1, \eps_1, \dots, i'_n, \eps_n) \text{     and     } 
h(\sigma) h_{\eps_1} w = v(i'_1, \eps_1, \dots, i'_n, \eps_n, i_{n+1}, \eps_{n+1}, \dots, i_{n+d}, \eps_{n+d}),$$ 
for some $i'_k \in I'_{-\eps_k}$. By our assumption, it follows that 
$$v \ = \ g v \ = \ h(\sigma) h_{\eps_1} v \ = \ v(i'_1, \eps_1, \dots, i'_n, \eps_n).$$ 
Thus, $i'_k = i_k$ for all $1 \leq k \leq n$, and therefore $g w = w$. 
\par
This concludes the proof.
\end{proof}

\begin{remark}
Theorems \ref{HNNC*simple} and \ref{HNNstructure} imply: \\ 
If either $\Sigma_{-1}$ or $\Sigma_1$ is non-amenable, then the amenablish radical of $\Lambda$ is trivial. \\ 
If $\Sigma_{-1}$ and $\Sigma_1$ are both amenable, then $\Lambda$ is amenablish. 
\end{remark}

{\bf Acknowledgments:} I want to thank Tron Omland for some useful suggestions.

\bibliographystyle{amsplain}

\begin{thebibliography}{99}

\bibitem{baumslag}
G.\ Baumslag.
\newblock Topics in Combinatorial Group Theory.
\newblock {\em Birkh\"{a}user}, 1993.

\bibitem{BH86}
E. Bedos,\ P.\ de~la~Harpe
\newblock Moyennabilité intérieure des groupes: définitions et exemples.
\newblock {\em Enseign. Math} (2) 32, no 1-2, 139-157, 1986.

\bibitem{BKKO}
Emmanuel Breuillard, Mehrdad Kalantar, Matthew Kennedy, and Narutaka Ozawa.
\newblock {$C^*$}-simplicity and the unique trace property for discrete groups.
\newblock {\em Publications mathématiques de l'IHÉS}, volume 126, pages 35–71 (2017). 

\bibitem{cohen}
D.\ Cohen.
\newblock Combinatorial Group Theory: a Topological Approach.
\newblock {\em Cambridge university perss}, 1989.

\bibitem{effros} 
E.\ Effros. 
\newblock Property $\Gamma$ and Inner Amenability. 
\newblock {\em Proceedings of the American Mathematical Society}, volume 47, No 2, pp. 483-486, February 1975.

\bibitem{leboudec16}
A.\ Le~Boudec.
\newblock Groups acting on trees with almost prescribed local action.
\newblock {\em Comment. Math. Helv.}, 91(2):253--293, 2016.

\bibitem{leboudec17}
A.\ Le~Boudec.
\newblock {$C^*$}-simplicity and the amenable radical.
\newblock {\em Invent. Math.}, 209(1):159--174, 2017.

\bibitem{BIO}
R. S. Bryder, N. A. Ivanov, T. Omland.
\newblock {$C^*$}-simplicity of HNN-extensions and groups acting on trees.
\newblock{arXiv:1711.10442} 

\bibitem{HO}
U.\ Haagerup and K.\ Knudsen~Olesen.
\newblock Non-inner amenability of the Thompson groups {$T$} and {$V$}.
\newblock {\em J. Funct. Anal.}, 272(11):4838--4852, 2017. 

\bibitem{Harpe}
Pierre de~la~Harpe.
\newblock On simplicity of reduced {$C^*$}-algebras of groups.
\newblock {\em Bull. Lond. Math. Soc.}, 39(1):1--26, 2007.

\bibitem{Harpepreaux}
P.\ de~la~Harpe and J.-P.\ Pr{\'e}aux.
\newblock {$C^*$}-simple groups: amalgamated free products, HNN-extensions, and fundamental groups of $3$-manifolds.
\newblock {\em J. Topol. Anal.}, 3(4):451--489, 2011.

\bibitem{IO}
N.~A.\ Ivanov and T.\ Omland.
\newblock {$C^*$}-simplicity of free products with amalgamation and radical classes of groups.
\newblock {\em J. Funct. Anal.}, 272(9):3712--3741, 2017.

\bibitem{KK}
M.\ Kalantar and M.\ Kennedy.
\newblock Boundaries of reduced {$C^*$}-algebras of discrete groups.
\newblock {\em Journal für die reine und angewandte Mathematik}, Volume 2017, Issue 727, Pages 247–267.

\bibitem{monodshalom}
N.\ Monod and Y.\ Shalom.
\newblock Cocycle superrigidity and bounded cohomology for negatively curved spaces.
\newblock {\em J. Differential Geom.}, 67(3):395--455, 2004.

\bibitem{MvN}
F. \ J. \ Murray and J. \ von Neumann. 
\newblock On rings of operators. IV. 
\newblock {\em Ann. of Math. (2)}, 44:716–808, 1943. 

\bibitem{Serre}
Jean-Pierre Serre.
\newblock Trees (translation of "Arbres, Amalgames, $SL_2$").
\newblock {Springer}, 2003.

\bibitem{Stalder}
Yves Stalder.
\newblock Moyennabilit\'e int\'erieure et extensions {HNN}.
\newblock {\em Ann. Inst. Fourier (Grenoble)}, 56(2):309--323, 2006.

\bibitem{vaes} 
S. \ Vaes. 
\newblock An inner amenable group whose von Neumann algebra does not have property Gamma. 
\newblock {\em Acta Math.}, 208(2): 389–394, 2012. 

\end{thebibliography}

\end{document}